\documentclass[a4paper,10pt,reqno, english]{amsart}

\usepackage{amsmath,amssymb,amscd,amsthm,amsfonts}
\usepackage{graphicx,subfigure}
\usepackage{hyperref}
\usepackage{dsfont}
\usepackage{color}

\newtheorem{theorem}{Theorem}[subsection]
\newtheorem*{theoremp}{Theorem}
\newtheorem{lemma}[theorem]{Lemma}
\newtheorem{claim}[theorem]{Claim}
\newtheorem{corollary}[theorem]{Corollary}
\newtheorem{example}[theorem]{Example}

\newtheorem*{problem}{Problem}
\newtheorem{definition}{Definition}

\def\rr{\mathds{R}}

\DeclareMathOperator{\conv}{conv}
\DeclareMathOperator{\vol}{vol}

\DeclareMathOperator{\diam}{diam}

\title{Quantitative combinatorial geometry for concave functions}

\author[Sarkar]{Sherry Sarkar}
\address{Georgia Institute of Technology, North Ave NW, Atlanta, GA 30332, United States} 
\email{ssarkar44@gatech.edu}

\author[Xue]{Alexander Xue}
\address{Cornell University, Ithaca, NY 14850, United States} 
\email{ajx3@cornell.edu}

\author[Sober\'on]{Pablo Sober\'on}\address{Baruch College, City University of New York, One Bernard Baruch Way, New York, NY 10010, United States} 
\email{pablo.soberon-bravo@baruch.cuny.edu}

\thanks{
This research project was done as
part of the 2019 CUNY Combinatorics REU, supported by NSF awards
DMS-1802059 and DMS-1851420.  Sober\'on's research is also supported by PSC-CUNY grant 62639.}

\begin{document}

\maketitle

\begin{abstract}
We prove several exact quantitative versions of Helly's and Tverberg's theorems, which guarantee that a finite family of convex sets in $\rr^d$ has a large intersection.  Our results characterize conditions that are sufficient for the intersection of a family of convex sets to contain a ``witness set'' which is large under some concave or log-concave measure.  The possible witness sets include ellipsoids, zonotopes, and $H$-convex sets.  Our results show that several new optimization problems can be solved with algorithms for LP-type problems.  We obtain colorful and fractional variants of all our Helly-type theorems.
\end{abstract}

\tableofcontents

\section{Introduction}

The study of the intersection patterns of convex sets is a substantial part of combinatorial geometry.  Helly's theorem and Tverberg's theorem are among the best-known results of this area.  Helly's theorem says that \textit{given a finite family of convex sets in $\rr^d$, if every $d+1$ or fewer sets have non-empty intersection, then the whole family has non-empty intersection} \cite{Helly:1923wr}.  Tverberg's theorem, on the other hand, says that \textit{given $(r-1)(d+1)+1$ points in $\rr^d$, there exists a partition of them into $r$ parts whose convex hulls intersect} \cite{Tverberg:1966tb}.  Many generalizations and extensions of Helly's and Tverberg's theorems have been proven, with classical examples including colorful, topological, and integer versions for both theorems \cite{Amenta:2017ed, Holmsen:2017uf, Blagojevic:2017bl, Barany:2018fy, DeLoera:2019jb}.

A particular family of generalizations of both theorems, called the \textit{quantitative} versions, give conditions that guarantee that the intersection of a family of convex sets in $\rr^d$ is large.  For example, we can ask for bounds on the volume of the intersection of a family of convex sets.

\begin{theoremp}[B\'ar\'any, Katchalski, Pach 1982 \cite{Barany:1982ga}]
	Let $\mathcal{F}$ be a finite family of convex sets in $\rr^d$.  If the intersection of every $2d$ or fewer sets in $\mathcal{F}$ has volume at least one, then the volume of $\cap \mathcal{F}$ is at least $d^{-2d^2}$.
\end{theoremp}

One can easily show that we cannot expect to conclude that the volume of $\cap \mathcal{F}$ is at least one if $d \ge 2$, so there is no \textit{exact} Helly theorem for the volume.  The lower bound for the volume of $\cap \mathcal{F}$ has been improved recently.  First by Nasz\'odi, giving a bound of $d^{-2d+o(1)}$ \cite{Naszodi:2016he} and then by Brazitikos, giving a bound of $d^{-3d/2 +o(1)}$ \cite{Brazitikos:2017ts}.  If we know that the intersection of subfamilies of larger cardinality, $\alpha d$ for some constant $\alpha$, have volume greater than or equal to one, Brazitikos showed that we can get a lower bound of $d^{-d+o(1)}$ for the volume of $\cap \mathcal{F}$ \cite{Brazitikos:2017ts}.  If one is willing to check much larger subfamilies, it was shown that we can get a bound of $1-\varepsilon$ on the volume of $\cap \mathcal{F}$ if we know that the intersection of every $\Theta(\varepsilon^{-(d-1)/2})$ sets has volume at least one \cite{DeLoera:2017gt}.

Quantitative Tverberg theorems are much more recent, and there are several interpretations of what the correct version should be.  We consider versions as in \cite{Soberon:2016co}.  In those variations, we replace points by convex sets, and we seek a partition of the family such that the intersection of the convex hulls of the parts is large.  As an analog for the B\'ar\'any-Katchalski-Pach theorem, we obtain the following example, which we prove in Section \ref{section-tverberg}.  The number of sets needed can be reduced slightly if $r$ is a prime power.

\begin{theorem}[Tverberg for volume]\label{theorem-tverbergexample} Let $r, d$ be positive integers and $\mathcal{F}$ be a family of $(r-1)\left(\frac{d(d+3)}{2}+1\right) + 1$ sets of volume one in $\rr^d$.  Then, there exists a partition of $\mathcal{F}$ into $r$ parts $\mathcal{A}_1, \ldots, \mathcal{A}_r$ such that the volume of $\bigcap_{j=1}^r \conv \left( \cup \mathcal{A}_j\right)$ is at least $d^{-d}$.
\end{theorem}

Quantitative Helly and Tverberg theorems have been proven for other continuous functions, such as diameter or surface area \cite{Brazitikos:2016ja, Soberon:2016co, Rolnick:2017cm}.  In both cases, we have an unavoidable loss, similar to their volumetric versions.  There are few cases for which there is an \textit{exact} quantitative theorem, such as a Helly theorem for inradius.  However, that case follows directly from Helly for containing translates of a set, which is a common exercise.  A version of Tverberg for the inradius as Theorem \ref{theorem-tverbergexample} also follows trivially from applying Tverberg's theorem to the set of centers of the incircles of the sets.  Exact quantitative Helly and Tverberg theorems have been proven for discrete functions over the convex sets, such as ``\textit{the number of points with integer coordinates in the set}'' \cite{Aliev:2016il, DeLoera:2017bl, Averkov:2017ge, DeLoera:2017th}.

In this manuscript, we present new families of quantitative Helly and Tverberg theorems that have exact versions for continuous functions.  Most of our theorems extend to colorful versions.  The simplest way to state our results is that we can obtain exact quantitative theorems for continuous functions as long as we impose conditions on the sets that witness the desired property.  For example, we obtain such theorems for the properties ``\textit{containing ellipsoids of large volume}'' and ``\textit{containing zonotopes of large Gaussian measure}''.

Quantitative Helly theorems can be considered as a bridge between combinatorial geometry and analytic convex geometry.  The results of Nasz\'odi and Brazitikos show how they are related to the sparsification of John decompositions of the identity \cite{Naszodi:2016he, Brazitikos:2017ts, Brazitikos:2016ja}. The results of De Loera, La Haye, Rolnick, and Sober\'on show how they are related to the theory of approximation of convex sets by polytopes \cite{DeLoera:2017gt}.  The results of Rolnick and Sober\'on show how the colorful versions are related to the analytic properties of ``floating bodies'' \cite{Rolnick:2017cm}.  We continue this trend and show how our exact quantitative Helly theorems are related to the study of concave functions and Minkowski sums.  Some of our results use the topological versions of Helly's theorem and of Tverberg's theorem in their proofs.  Topological methods have not been used before for quantitative variations.

Our results depend on two main components: the function we work with, and the family of sets we use to witness that we achieve a desired value in the function.  The Helly numbers (i.e., the size of the subfamilies we must check) in our results are determined by the dimension of the space of possible witness sets, and they are often optimal.  Our Tverberg theorems have a similar dependence.  This gives an intuitive idea of why the loss of volume is unavoidable in the B\'ar\'any-Katchalski-Pach theorem: the space of convex sets in $\rr^d$ has infinite dimension.  We obtain results for a wide range of functions.  It's important to note that just finding good families of witness sets is not enough.  Otherwise, we would be able to obtain exact quantitative results for the diameter, as a segment always realizes it.  This would contradict the examples presented previously by the third author \cite{Soberon:2016co}.  We do obtain some exact quantitative results for the diameter under $\ell_1$-norm instead of $\ell_2$-norm, which we discuss in sections \ref{secion-more-ellipsoids} and \ref{section-affinetverberg}.

Our results can be split into two groups:

\begin{itemize}
	\item \textbf{Results with a geometric proof.}  Several of our results can boil down to standard combinatorial geometry theorems in higher-dimensional spaces.  For our parametrizations to work, we need strong conditions on the sets that witness a large intersection.  These results apply to a large family of functions, which includes all log-concave measures in $\rr^d$.  Moreover, in the cases that this framework applies, we get versions of almost every known variation of Helly and Tverberg's theorems, including quantitative $(p,q)$-type results \cite{Alon:1992gb}.
	\item \textbf{Results with a topological proof.} A simple contractibility argument allows us to reduce many quantitative Helly-type results to Kalai and Meshulam's topological colorful Helly theorem \cite{Kalai:2005tb}.  These results apply to a broad family of possible witness sets, at the cost of a reduced family of functions.  The topological properties of the spaces of witness sets allow us to obtain smaller Helly numbers.  The related Tverberg-type results can be proved with the topological version of Tverberg's theorem \cite{Barany:1981vh, Oza87, Volovikov:1996up}.  In those cases, we require some parameters to be prime powers.
\end{itemize}

Both cases are general enough to contain the volume as the target function.  We show that the topological colorful Helly theorem by Kalai and Meshulam has applications to purely geometric Helly-type problems.  This had been observed before for Carath\'eodory-type theorems \cite{Holmsen:2017iw}.  We first prove all our Helly-type results in sections \ref{section-geometric} and \ref{section-topological}.  Then, we show how the methods extend to Tverberg's theorem in Section \ref{section-tverberg}.

We present some volumetric Helly theorems in this section since they are the easiest to compare with previous quantitative Helly theorems.  First, let us introduce matroids.  There are plenty of equivalent definitions for matroids \cite{Oxley:2006uz}.  Given a finite set $V$ of vertices, we say a \textit{matroid} or matroidal complex $M$ on $V$ is a family of subsets of $V$ with three properties.
\begin{itemize}
	\item $\emptyset \in M$.
	\item If $A \subset B$ and $B \in M$, then $A \in M$.
	\item If $A, B \in M$ and $B$ has more elements than $A$, there exists an element $a \in B\setminus A$ such that $A \cup \{a\} \in M$.
\end{itemize}

We call the sets in $M$ \textit{independent}.  For a subset $V' \subset V$, we denote by $\rho(V')$ the rank of $V'$, which is the cardinality of the largest independent set contained in $V'$.

\begin{theorem}[Matroid Helly for ellipsoids of volume one]\label{theorem-colorfulhellyellipsoids}
Let $M$ be a matroid on a set $V$ of vertices with rank function $\rho$. For each $v$ in $V$, we are given a convex set $F_v$ in $\rr^d$.  We know that for each set $V' \in M$, there exists an ellipsoid of volume one contained in $\cap_{v \in V'} F_v$.  Then, there exists a set $\tau \subset V$ such that $\rho (V \setminus \tau) \le d(d+3)/2-1$ and for which there exists an ellipsoid of volume one contained in $\cap_{v \in \tau} F_v$.
\end{theorem}

The result above can be extended further.  We present a generalization in Theorem \ref{theorem-strong-matroid-affine} which is highly malleable.  We use it to show variations of Theorem \ref{theorem-colorfulhellyellipsoids} in sections \ref{subsection-affine-images-revisited} and \ref{secion-more-ellipsoids}.  Those have different Helly numbers, depending on restrictions to the space of ellipsoids considered.  We also have versions for minimal enclosing ellipsoids, or for taking the sum of the lengths of the axes instead of the volume.  If we pick a partition matroid in Theorem \ref{theorem-colorfulhellyellipsoids}, we obtain the following colorful version.

\begin{corollary}[Colorful Helly for ellipsoids of volume one]\label{corollary-colorfulhellyellipsoids}
	Let $n=\frac{d(d+3)}{2}$ and $\mathcal{F}_1, \ldots, \mathcal{F}_n$ be finite families of convex sets in $\rr^d$.  Suppose that for every choice $F_1 \in \mathcal{F}_1, \ldots, F_n \in \mathcal{F}_n$, their intersection contains an ellipsoid of volume one.  Then, there exists an index $i \in \{1, \ldots, n\}$ such that $\cap \mathcal{F}_i$ contains an ellipsoid of volume one.  Moreover, if $n= \frac{d(d+3)}{2}-1$, the conclusion of the theorem may fail.

\end{corollary}

Corollary \ref{corollary-colorfulhellyellipsoids} has been proved by Dam\'asdi \cite{Damasdi:2017ta} using methods similar to those shown by De Loera, La Haye, Oliveros, and Rold\'an-Pensado \cite{DeLoera:2017th}.  The result above exemplifies why the name colorful is attributed to these variations, as we can consider each $\mathcal{F}_i$ as a family of sets painted with the same color.  We have not seen how the geometric methods of Dam\'asdi can be extended to ``colorings'' by matroidal complexes.

  Theorem \ref{theorem-colorfulhellyellipsoids} implies a colorful Helly theorem for the volume similar to the B\'ar\'any-Katchalski-Pach theorem.  Inscribed ellipsoids of maximal volume, called John ellipsoids, have been studied extensively in classical convex geometry \cite{Ball:1997ud}.  In particular, for a convex set $K \subset \rr^d$ with non-empty interior whose John ellipsoid $\mathcal{E}$ is centered at the origin we have
\[
\mathcal{E} \subset K \subset d \mathcal{E}.
\]

This implies that $\operatorname{vol}(\mathcal{E}) \ge d^{-d} \operatorname{vol}(K).$  We can use this fact in conjunction with Corollary \ref{corollary-colorfulhellyellipsoids} to prove a colorful Helly theorem for the volume.  However, we can get a stronger result.  The following theorem is obtained by using Theorem \ref{theorem-colorfulhellyellipsoids} to bootstrap the results by Brazitikos \cite{Brazitikos:2017ts}.

\begin{theorem}\label{theorem-truecolorful-volumentric}
	Let $M$ be a matroid on a set $V$ of vertices with rank function $\rho$, and let $d$ be a positive integer.  For each $v \in V$ we are given a convex set $F_v$ in $\rr^d$.  We know that for each independent set $V' \subset V$ of at most $2d$ vertices, $\cap_{v \in V'} F_v$ has volume at least one.  Then, there exists a set $\tau \subset V$ of vertices such that $\rho ( V \setminus \tau) \le \frac{d(d+3)}{2}-1$ and the volume of $\cap_{v \in \tau} F_v$ is at least $d^{-3d/2-o(1)}$.
\end{theorem}

Again, a more familiar statement comes from the application of the theorem above to a particular partition matroid.  This implies the following corollary, which follows the style of Lov\'asz's colorful Helly theorem \cite{Barany:1982va}.

\begin{corollary}\label{corollary-simpleellipsoidcoloring}
	Let $n=\frac{d(d+3)}{2}$ and $\mathcal{F}_1, \ldots, \mathcal{F}_n$ be finite families of convex sets in $\rr^d$, considered as color classes.  Suppose that for every choice $F_1, \ldots, F_{2d}$ of $2d$ convex sets of different colors, their intersection has volume greater than or equal to one.  Then, there exists an index $i \in \{1, \ldots, n\}$ such that $\cap \mathcal{F}_i$ has volume greater than or equal to $d^{-3d/2-o(1)}$.
\end{corollary}

Notice that if $\mathcal{F}_1 = \ldots = \mathcal{F}_n$ we recover the best known bound for Helly's theorem for the volume \cite{Brazitikos:2017ts}. Intuitively, the colorful versions we could prove previously dictate the number of color classes needed, yet the known ``monochromatic'' versions dictate the size of the subfamilies we need to check.  We do not know if the value of $n$ in Corollary \ref{corollary-simpleellipsoidcoloring} (or the value of $\rho(V\setminus \tau) + 1$ in Theorem  \ref{theorem-truecolorful-volumentric}) can be reduced from $d(d+3)/2$ to $2d$.  If the volume obtained in the final ellipsoid is allowed to be $\sim d^{-3d^2}$, it was recently shown by Dam\'asdi, F\"oldv\'ari, and Nasz\'odi that the number of color classes can be reduced to $3d$ \cite{Damasdi:2019vm}.%The methods used by Naszodi and Brazitikos rely on a lemma by Dvoretzky and Rogers regarding John decompositions of the identity \cite{Dvoretzky:1950kt}, or sparsifications of John decompositions of the identity \cite{Batson:2012fc}.  

%In contrast, the methods of the proof of Theorem \ref{theorem-colorfulhellyellipsoids} rely on the nerve lemma and the topological colorful Helly theorem of Kalai and Meshulam \cite{Kalai:2005cm}.

We also obtain quantitative Helly theorems for zonotopes and $H$-convex sets instead of ellipsoids.  We describe here one of the results for zonotopes.  Given directions $v_1, \ldots, v_k \in \rr^d \setminus \{0\}$, and $p \in \rr^d$, we say that a convex set $K$ is a zonotope centered at $p$ with directions $v_1, \ldots, v_k$ if $K$ is the Minkowski sum of $k$ segments with directions in $v_1, \ldots, v_k$.  In other words, there exist $\alpha_1, \ldots, \alpha_k \ge 0$ such that
\[
K = p + \Big( (\alpha_1 v_1) \oplus (\alpha_2 v_2) \oplus \ldots \oplus (\alpha_k v_k)\Big),
\]
where $\oplus$ stands for the Minkowski sum.

\begin{theorem}\label{theorem-zonotopes}
	Let $k \ge d$ be positive integers and $v_1, \ldots, v_k$ be directions in $\rr^d$.  Let $\mathcal{F}$ be a finite family of convex sets in $\rr^d$.  If the intersection of every $k+d$ sets in $\mathcal{F}$ contains a zonotope with directions $v_1, \ldots, v_k$ that has volume one, then $\cap \mathcal{F}$ contains a zonotope with directions $v_1, \ldots, v_k$ with volume one.
\end{theorem}

Since the proof of Theorem \ref{theorem-zonotopes} relies on a reduction to Helly's theorem, we get for free matroid, colorful, fractional, and $(p,q)$ versions of the theorem above. 

 We discuss the fractional and $(p,q)$ versions of our theorems in Section \ref{section-remarks}, along with open problems and possible directions of research.  In Section \ref{section-geometric} we present our Helly results that have geometric proofs, and in Section \ref{section-topological} we present our Helly results that have topological proofs.  In Section \ref{section-tverberg} we present our Tverberg results.

\section{Helly results with a geometric proof}\label{section-geometric}

The goal of this section is to present several quantitative Helly-type theorems which can be reduced to a standard Helly theorem in higher dimensions.  We can achieve this when we have the following two ingredients.

\begin{itemize}
	\item \textbf{A class $\mathcal{C}$ of sets that is easy to parametrize.}  In most cases, we want families that are closed under Minkowski sum: if $A, B \in \mathcal{C}$, then $A\oplus B \in \mathcal{C}$.  However, we also present results for families of convex sets which are not closed under Minkowski sum.  The dimension of $\mathcal{C}$ as a topological space with the Hausdorff metric is going to determine our Helly numbers.  The parametrization should give a convex structure to $\mathcal{C}$.
	\item \textbf{A $\min$-concave function $f:\mathcal{C} \to \rr$.} We say that a function is $\min$-concave if $f(\lambda A + (1-\lambda) B) \ge \min\{f(A), f(B) \}$ for all $A, B \in \mathcal{C}$ and $\lambda \in [0,1]$.  The definition of convex combination $\lambda A + (1-\lambda)B$ depends on our parametrization of $\mathcal{C}$.  In many cases it represents $\lambda A \oplus (1-\lambda)B$, which we refer to as a Minkowski convex combination of $A$ and $B$.
\end{itemize}

We first present a very general Helly theorem.  This result will work as a blueprint for our results with geometric proofs.  For some applications, we show how to reduce the resulting Helly number.

\begin{definition}
We say that a family $\mathcal{C}$ of convex sets in $\rr^d$ has a Minkowski parametrization in $\rr^l$ if there exists a surjective function $D: \rr^l \to \mathcal{C}$ for which $ D(\lambda a+ (1-\lambda) b) = (\lambda D(a)) \oplus ((1-\lambda) D(b))$ for all $a,b \in \rr^l$ and all $\lambda \in [0,1]$.
\end{definition}

With this, we can now state the following Helly-type theorem.

\begin{theorem}\label{theorem-general-helly}
	Let $\mathcal{C}$ be a family of convex sets in $\rr^d$ that has a Minkowski parametrization in $\rr^l$.  Let $f: \mathcal{C} \to \rr$ be a function such that $f(\lambda A \oplus (1-\lambda)B) \ge \min \{f(A), f(B)\}$ for all $A, B \in \mathcal{C}$ and $\lambda \in [0,1]$.  Then, for any finite family $\mathcal{F}$ of convex sets in $\rr^d$, if the intersection of every $l+1$ or fewer sets in $\mathcal{F}$ contains a set $K \in \mathcal{C}$ such that $f(K) \ge 1$, then $\cap \mathcal{F}$ contains a set $K \in \mathcal{C}$ such that $f(K) \ge 1$.
\end{theorem}

We can immediately make two observations.  First, we can apply almost any generalization or extension of Helly's theorem in $\rr^l$ and obtain new results in $\rr^d$.  The proof method above gives quantitative colorful, fractional, and $(p,q)$ theorems in $\rr^d$.  The second observation is that, even though $\mathcal{C}$ is $l$-dimensional, the set $\{ C \in \mathcal{C} : f(C) = 1\}$ is in general $(l-1)$-dimensional.  This makes it possible to reduce the Helly number from $l+1$ to $l$ in several cases, even if $f$ is not a linear function on $\mathcal{C}$.  For some cases, such as zonotopes, we show the improvement in this section.  For other cases, such as ellipsoids, we require the topological proofs.

Let us mention two examples of functions to which we can apply our methods.

\begin{example}[Log-concave measures]\label{example-concave-measures}
	We say a measure $\mu$ in $\rr^d$ is log-concave if $\mu(\lambda A \oplus (1-\lambda)B) \ge \mu(A)^{\lambda} \mu(B)^{1-\lambda}$, for any two Borel sets $A, B$ and $\lambda \in [0,1]$.  There is a simple way to obtain log-concave functions \cite{Borell:1975kx}.  It suffices to take a log-concave density function $p: \rr^d \to \rr^+$ and consider
	\[
	\mu(A) = \int_A p.
	\]
	There are abundant log-concave density functions $p$ to choose from \cite{Saumard:2014fz}.  Common examples are $p$ being constant (which gives $\mu$ as the volume), $p = e^{-\psi}$, where $\psi$ is any convex function (which makes $\mu$ a Gaussian measure if $\psi (x) = ||x||^2$), or $p$ being a multivariate real stable polynomial (if we restrict our sets to the points in $\rr^d$ with positive coordinates).
\end{example}

\begin{example}[Simultaneous approximation of convex sets by a single set]
	Given two sets $A, K \subset \rr^d$, and a positive real number $\varepsilon$, we say that a translation of $A$ is an $\varepsilon$-approximation of $M$ with center $a \in \rr^d$ such that
	\[
	a+ A \subset K \subset a+ (1+\varepsilon)A.
	\]
	Notice that if $\lambda \in [0,1]$ and $A, B$ are $\varepsilon$-approximations of the same convex set $K$, with centers $a, b$, respectively, then $\lambda A \oplus (1-\lambda)B$ is also an $\varepsilon$-approximation of $K$ with center $\lambda a + (1-\lambda)b$.  We say that a translation of $A$ simultaneously $\varepsilon$-approximates a family $\mathcal{F}$ of sets if there exists an $a \in \rr^d$ such that the condition above holds simultaneously for all $K \in \mathcal{F}$ (i.e., we use the same translation vector for all sets in $\mathcal{F})$.
	\end{example}

Let $K$ be a convex set and $\varepsilon >0$.  If $\mathcal{C}_0$ denotes the bounded convex sets in $\rr^d$ whose barycenter is at the origin, we can define a $\{0,1\}$-function.
\begin{align*}
	f_K: \rr^d \times \mathcal{C}_0 & \to \{0,1\} \\
	(a,A) & \mapsto \begin{cases}
		1 & \mbox{if } a + A \subset M \subset a + (1+\varepsilon)A \\
		0 & \mbox{otherwise}
	\end{cases}
\end{align*}

The function above is min-concave with respect to the Minkowski sum because support functions respect Minkowski convex combinations.  In formal terms, if $h_v(C)$ is the support function in direction $v$ of a set $C$ and $\lambda \in [0,1]$, then
	\[
	h_v (\lambda A \oplus (1-\lambda)B) = \lambda h_v(A) + (1-\lambda) h_v(B).
	\]
	
Let us show how this family of functions give exact quantitative Helly theorems, even though they do not fit exactly into the framework of Theorem \ref{theorem-general-helly} (since we have a different function for each convex set $K$).

\begin{theorem}\label{theorem-helly-aproximation}
	Let $\varepsilon > 0$ and $\mathcal{C}$ be a family of convex sets in $\rr^d$, that has a Minkowski parametrization in $\rr^l$.  For a finite family $\mathcal{F}$ of convex sets in $\rr^d$, we know that for every subfamily $\mathcal{F}'$ of $l+d+1$ or fewer sets of $\mathcal{F}$ there exists a translation of a set in $\mathcal{C}$ that is a simultaneous $\varepsilon$-approximation for all sets in $\mathcal{F}'$.  Then, there exists a translation of a set in $\mathcal{C}$ that is a simultaneous $\varepsilon$-approximation all sets in $\mathcal{F}$.
	\end{theorem} 	
	
	\begin{proof}
		Let $D: \rr^l \to \mathcal{C}$ be the Minkowski parametrization of $\mathcal{C}$.  Given $K \in \mathcal{F}$, let $f_K : \rr^d \times \mathcal{C} \to \{0,1\}$ be defined as above.  For each $K \in \mathcal{F}$, let $S(K) = \{(a,x) \in \rr^d \times \rr^l : f_K (a,D(x)) = 1\}$.  Since $F \subset \rr^d$ is convex, then $S(F) \subset \rr^{l+d}$ is also convex.  Therefore, an application of Helly's theorem in $\rr^{l+d}$ gives us the desired result.
	\end{proof}
	
	If we are interested in Theorem \ref{theorem-helly-aproximation} for a particular family of convex sets, such as axis-parallel boxes, it suffices to consider with $\mathcal{C}$ the axis-parallel boxes centered at the origin, since the translation vector will be added afterward.  This reduces the Helly number obtained by $d$.  Let us now analyze families of convex sets that admit parametrizations that we can use in Theorem \ref{theorem-general-helly}.

\subsection{Zonotopes and cylinders with fixed directions}

Let ${v}_1, \ldots, v_k$ be fixed directions in $\rr^d$.  For a vector $\bar{m}=(p_1, \ldots, p_d, \alpha_1, \ldots, a_k) \in \rr^{k+d}$ where $\alpha_i \ge 0$ for all $1\le i \le k$, we can consider the zonotope
\[
Z(\bar{m}) = p + \Big((\alpha_1 v_1) \oplus (\alpha_2 v_2) \oplus \ldots \oplus (\alpha_k v_k)\Big),
\]
where $p = (p_1, \ldots p_d) \in \rr^d$.  Then, not only is the family of all zonotopes defined this way closed under Minkowski sums, but $Z(\bar{l}) \oplus Z(\bar{m}) = Z(\bar{l} + \bar{m})$ for all $\bar{l}, \bar{m} \in \rr^{k+d}$.

If we apply Theorem \ref{theorem-general-helly} to zonotopes and min-concave measures, we obtain a Helly number of $k+d+1$.  We now show how to reduce this number by one.

\begin{theorem}\label{theorem-zonotopes-strong}
	Let $v_1, \ldots, v_k$ be directions in $\rr^d$ and $\mathcal{C}$ be the family of zonotopes with directions $v_1, \ldots, v_k$ in $\rr^d$.  Let $\mu$ be a $\min$-concave measure in $\rr^d$.  Then, given a finite family $\mathcal{F}$ of convex sets in $\rr^d$, if the intersection of every $k+d$ or fewer sets in $\mathcal{F}$ contains a zonotope $K \in \mathcal{C}$ such that $\mu(K) \ge 1$, then $\cap \mathcal{F}$ contains a zonotope $K \in \mathcal{C}$ such that $\mu(K) \ge 1$.
\end{theorem}

\begin{proof}
	We modify a bit the proof of Theorem \ref{theorem-general-helly} to fit this theorem.  Given a convex set $F$ in $\rr^d$, we consider $S(F) \subset \rr^{k+d-1}$ the set of points 
	\[
	(p_1, \ldots, p_d, \alpha_1, \ldots, \alpha_{k-1}) \in \rr^{k+d-1}
	\]
	such that $\alpha_i \ge 0$ for $i =1,\ldots, k-1$ and the following condition holds.  There exists a value of $\alpha \ge 0$ for which the zonotope
	\[
	Z = p + \Big( (\alpha_1 v_1) \oplus (\alpha_2 v_2) \oplus \ldots \oplus (\alpha_{k-1} v_{k-1}) \oplus (\alpha v_k)\Big)
	\]
	is contained in $F$ and $\mu(Z) \ge 1$.  We only have to show that $S(F)$ is convex.  In order to do this, consider two points $ \bar{a}=(p_1, \ldots, p_d, \alpha_1, \ldots, \alpha_{k-1}), \bar{b}=(q_1, \ldots, q_d, \beta_1, \ldots, \beta_{k-1})$ in $S(F)$.  They each have a value $\alpha, \beta$ such that $(\bar{a}, \alpha), (\bar{b}, \beta)$ describe zonotopes contained in $F$ with large $\mu$-measure.  For $\lambda \in [0,1]$, consider the zonotope described by $(\lambda \bar{a} + (1-\lambda) \bar{b}, \lambda \alpha + (1-\lambda) \beta)$.  This is a Minkowski convex combination of the two zonotopes $Z_1, Z_2$ described by $(\bar{a}, \alpha), (\bar{b}, \beta)$, respectively, so it is contained in $F$.  Moreover, by the $\min$-concavity of $\mu$, we have that since $\mu(Z_1) \ge 1$ and $\mu(Z_2) \ge 1$, then $\mu(\lambda Z_1 \oplus (1-\lambda)Z_2) \ge 1$.  Therefore, the first $k+d-1$ coordinates representing $\lambda Z_1 \oplus (1-\lambda)Z_2$ describe a point of $S(F)$, as we wanted.
\end{proof}

The reader may notice that the restriction of $\mu = \operatorname{vol}$ shows that Theorem \ref{theorem-zonotopes} is a corollary of Theorem \ref{theorem-zonotopes-strong}.  Moreover, the convexity of $S(F)$ allows us to obtain colorful, fractional, and $(p,q)$ versions of Theorem \ref{theorem-zonotopes-strong}.

If we take $k=d$ and $v_1, \ldots, v_d$ the canonical basis of $\rr^d$, we obtain a Helly theorem for axis-parallel boxes of volume one.  It should be noted that recovering volumetric Helly theorems for which the Helly number is $2d$ is of particular interest since this is the Helly number in the classical results of B\'ar\'any, Katchalski, and Pach \cite{Barany:1982ga}.

\begin{example}[generalized zonotopes in $\rr^d$]
	In the definition of a zonotope, we can replace any term $\alpha_k v_k$ by $\alpha_k T_k$, where $T_k$ is a convex set in some $r$-dimensional subspace of $\rr^d$.  We call this a generalized zonotope.  For convex sets and $\alpha, \beta >0$ we have
	\[
	\alpha T_k \oplus \beta T_k = (\alpha + \beta) T_k,
	\]
	 so the sets constructed this way are still closed under Minkowski sums.  We do not need to increase the Helly number beyond $k+d$ for Helly-type theorems involving generalized zonotopes.  %Notice that if $k=2$$T_k$ is a ball, then we can obtain cylinders.
\end{example}

As an example of a generalized zonotope, consider the generalized zonotopes in $\rr^d$ of the form
\[
p + (\alpha_1 v_1 \oplus \alpha_2 T_2),
\]
where $v_1$ is a direction in $\rr^d$ and $T_2$ is a unit ball in a hyperplane orthogonal to $v_1$.  These are parallel cylinders in $\rr^d$.  In the case $d=3$, this gives us a Helly theorem for containing a vertical cylinder of volume one with Helly number five.  A direct application of Theorem \ref{theorem-helly-aproximation} shows that \textit{if for every $6$ sets of a finite family of convex sets in $\rr^3$ there exists a vertical cylinder such that one of its translates is a simultaneous $\varepsilon$-approximation of them, then there exists a vertical cylinder such that one of its translates is an $\varepsilon$-approximation of the whole family}.

%{\LARGE{\textcolor{red}{TALK ABOUT WELZL'S ALGORITHMS}}}

%\begin{theorem}\label{theorem-zonotope-approximation}
%	Let $\mathcal{F}$ be a family of $n$ convex sets in $\rr^d$, and $v_1, \ldots, v_k$ be $k$ different directions in $\rr^d$.  Assume that $k, d$ are fixed.  Then, the problem of determining the smallest $\varepsilon$ such that there exists a zonotope with directions $v_1, \ldots, v_k$ that is a simultaneous $\varepsilon$-approximation for $\mathcal{F}$ can be solved in time $O(n^{k+d})$.
%\end{theorem}

%\begin{proof}
%	Using Theorem \ref{theorem-approximation-by-zonotopes}, we only need to find this $\varepsilon$ for every $(k+d)$-tuple of sets in $\mathcal{F}$.  The maximum $\varepsilon$ we find this way will be the value we look for.  Therefore, we only need to solve $O(n^{k+d})$ problems, each of which can be solved in time independent of $n$.
%\end{proof}

\subsection{Affine images of a fixed set}\label{example-matrices}
	Consider the set
	\begin{align*}
	\mathcal{C} & = \{(a,A): a \in \rr^d, \ A \mbox{ is a symmetric $d\times d$ positive definite matrix}\}.	\end{align*}
		
	We can assign the linear structure of $(a, A)$ as points in $\rr^{d(d+3)/2}$ to $\mathcal{C}$.
		Let $K \subset \rr^d$ be a set.  Given a convex set $M \subset \rr^d$, we also consider
	\[
	S_K(M) = \{(a,A) \in \mathcal{C} : a +AK \subset M\}.
	\]
	  Gruber proved that for the case $K = B_d$, the unit ball in $\rr^d$, and for any convex set $M \subset \rr^d$, the set $S_{B_d}(M)$ is convex \cite{Gruber:2008gr}.  We can give a short one-line proof of this fact for any $K$.  For two points $(a, A)$, and $(b, B)$ in $S_K(M)$, and $\lambda \in [0,1]$, we have
	\[
	[\lambda a + (1-\lambda)b] + [\lambda A + (1-\lambda)B] K \subset \lambda (a + AK) \oplus (1-\lambda) (b + BK) \subset \lambda M \oplus (1-\lambda)M = M.
	\]
	Therefore, $S_K(M)$ is convex.

Notice that the family of sets of the form $a+AB_d$ with $(a,A) \in \mathcal{C}$ parametrizes all ellipsoids in $\rr^d$.  To see this, consider an ellipsoid centered at the origin, $XB_d$ where $X$ is any non-singular matrix.  We can find a polar decomposition of $X = AQ$, where $A$ is a symmetric positive definite matrix and $Q$ is orthogonal.  Then, $XB_d = AQB_d = AB_d$.

 Even though this is not a Minkowski parametrization of ellipsoids, the fact that $S_K(M)$ is convex is enough to prove an analog version of Theorem \ref{theorem-general-helly} for min-concave functions on symmetric matrices.  Doing this is almost enough to prove Corollary \ref{corollary-colorfulhellyellipsoids}.  We need to consider the function $\det (A)$, which is log-concave in $\mathcal{P}_d$, the space of positive definite $d \times d$ matrices.  However, this application requires the number of color classes to be $n = \frac{d(d+3)}{2}+1$.  In order to reduce the dimension by one, we will use the topological methods of the next section.  We have not found a way to reduce this Helly number in a similar way to the proof of Theorem \ref{theorem-zonotopes-strong}.  An advantage of this slightly weaker form of Corollary \ref{corollary-colorfulhellyellipsoids} is that we do get a $(p,q)$ theorem for containing ellipsoids of volume one.  We can, however, reduce the dimension for other functions, as in the next theorem.
 
 \begin{theorem}
 	Let $\mathcal{F}$ be a finite family of convex sets in $\rr^d$.  If the intersection of every $\frac{d(d+3)}{2}$ or fewer sets in $\mathcal{F}$ contains an ellipsoid whose sum of lengths of principal axes is equal to one, then $\cap \mathcal{F}$ contains an ellipsoid whose sum of lengths of principal axes is equal to one. 
 \end{theorem}
 
 \begin{proof}
 	Let $\tau$ be the diameter of $B_d$ and $\operatorname{tr}(\cdot)$ denote the trace function.  For each $K \in \mathcal{F}$, let 
 	\[
 	S(K) = \left\{(a,A) \in \rr^d \times \mathcal{P}_d : \ \operatorname{tr}(A) = \frac{1}{d \tau}, \quad a +AB_d \subset K \right\}.
 	\]
 	Since the trace is a linear function, $S(K)$ is a subset of a $\left(\frac{d(d+3)}{2}-1\right)$-dimensional affine subspace of $\rr^d \times \mathcal{P}_d$.  Moreover, $S(K)$ is convex.  The lengths of the axes of the ellipsoids $a +AB_d$ are equal to $\tau$ times the eigenvalues of $A$, so $d \tau \operatorname{tr}(A)$ is the sum of the lengths of the axes of the ellipsoid.  An application of Helly's theorem in this $\left(\frac{d(d+3)}{2}-1\right)$-dimensional affine subspace of $\rr^d \times \mathcal{P}_d$ gives us the desired result.
 \end{proof}
 
%We can reduce the dimension for other functions.  For example, the trace is a linear function, so restricting the set $S_{B_d}(M)$ to the pairs $(a,A)$ where $\operatorname{tr} (A) = 1$ still gives us a convex set.  Notice that $\operatorname{tr}(A)$ is equal to $2d$ times the sum of  the lengths of the axis of $a + AB_d$.  Moreover, the space of pairs $(a,A) \in \mathcal{C}$ for which $\operatorname{tr}(A)=1$ is of dimension $\frac{d(d+3)}{2}-1$.

\begin{figure}[h]
	\centerline{\includegraphics{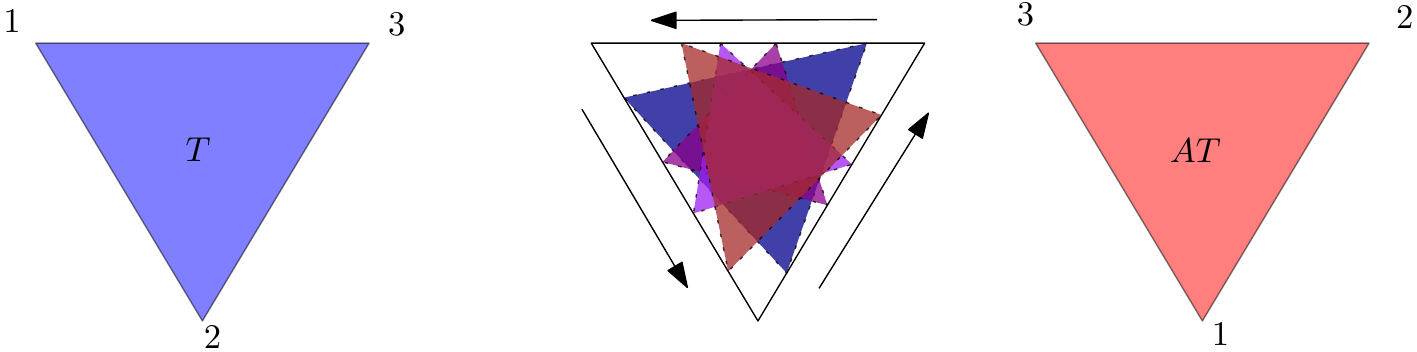}}
	\caption{The matrix $A$ preserves the triangle $T$ setwise, but permutes the vertices.  In the space of triangles, we cannot deform $T$ to $AT$ by using triangles of constant area and contained in $T$.}\label{figure-triangles}
\end{figure}

Some readers may find the condition in Section \ref{example-matrices} of the matrices being positive semidefinite unusual.  However, we cannot remove it completely.  Consider $\mathcal{T}$ to be the set of triangles in the plane, and $K$ be a particular area one triangle in the plane.  Then $\mathcal{T}$ can be parametrized as the family of sets of the form $a + AK$ where $a \in \rr^2$ and $A$ is a non-singular $2\times 2$ matrix.  However, the set $S_{\triangle}(K) \subset \mathcal{T}$ that represents triangles of area at least one which are contained in $K$ consists of six isolated points (one for each permutation of vertices).  In order for $S_{\triangle}(K)$ to be convex, or even just connected, we need some conditions on the matrices involved.  Figure \ref{figure-triangles} illustrates this argument.

%Let us look at Theorem \ref{theorem-zonotopes} in the case $k=d$.  We may assume that the $d$ directions are mutually perpendicular, so our witness sets are boxes in $\rr^d$.  The same examples that B\'ar\'any, Katchalski, and Pach used to show that they needed to check the volume of the intersection of every $2d$ sets in order to 

\subsection{$H$-convex sets}

	Let $S^{d-1} \subset \rr^d$ be the unit sphere.  Given a family of directions $H \subset S^{d-1}$, not contained in any closed half-sphere of $S^{d-1}$, we consider all half-spaces of the form $\{ y : \langle y, h \rangle \le \lambda\}$ where $h \in H, \lambda \in \rr$, and $\langle \cdot, \cdot \rangle$ stands for the dot product.  We refer to these halfspaces as \textit{$H$-halfspaces}.  We say that a set $Y \subset \rr^d$ is $H$-convex if and only if it is the intersection of a set of $H$-halfspaces.  Boltyanski and Martini characterized the sets $H$ for which $H$-convex sets are closed under Minkowski sums \cite{Boltyanski:2003ir}.
	
	\begin{figure}[h]
	\centerline{\includegraphics[scale=0.8]{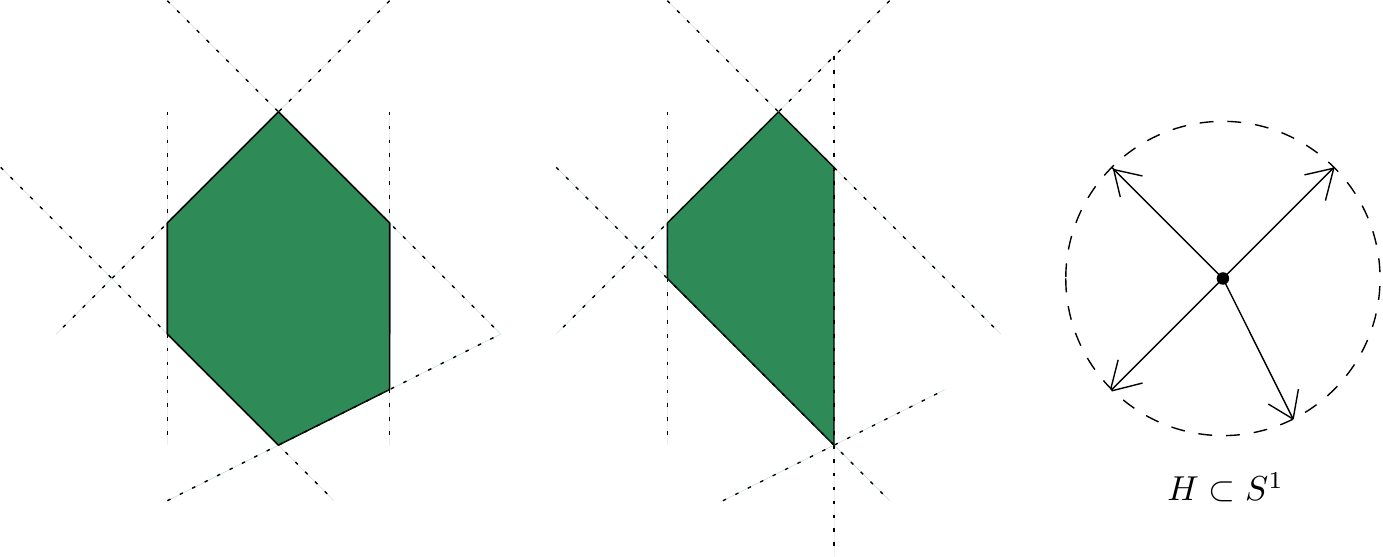}}
	\caption{Two $H$-convex sets in the plane.  Redundant half-spaces are used as support hyperplanes.}
\end{figure}
	
	The intuitive idea is to consider two different $H$-convex polytopes $P$ and $Q$ in $\rr^d$.  For each $1\le k \le d-2$, it is possible that a face in the $k$-dimensional skeleton of $P$ and a face in the $(d-k-1)$-skeleton of $Q$ get added to make a facet of $P\oplus Q$.  The condition of Boltyanski and Martini is that the direction orthogonal to this new facet is contained in $H$, for all such possible directions.  In the plane, any finite set $H \subset S^1$ not contained in any closed half-circle gives a set of directions for which $H$-convex sets are closed under Minkowski sums.  We avoid stating this condition explicitly since we don't use it directly.

\begin{claim}
	Let $H \subset S^{d-1}$ be a finite set of directions for which $H$-convex sets are closed under Minkowski sums.  Then, the family $\mathcal{C}$ of $H$-convex sets admits a Minkowski parametrization in $\rr^{|H|}$.
\end{claim}

\begin{proof}
	Let $H=\{v_1, \ldots, v_{|H|}\}$.  Any $H$-convex set $K$ can be parametrized by a vector $(\lambda_1, \ldots, \lambda_{|H|})) \in \rr^{|H|}$ such that $\lambda_i = h_{v_i}(K)$.  In other words, $\lambda_i$ is the support function of $K$ in the direction $v_i$.  Notice that if vectors $\bar{a}, \bar{b}$ parametrize sets $A$ and $B$ and $\lambda \in [0,1]$, then the vector $\lambda \bar{a} + (1-\lambda) \bar{b}$ parametrizes $\lambda A \oplus (1-\lambda) B$. This is the Minkowski parametrization we wanted.
\end{proof}

We can apply Theorem \ref{theorem-general-helly} to families of $H$-convex sets as long as they are closed under Minkowski sums.  The Helly numbers will be $|H|+1$.  However, for log-concave measures, we can again decrease the Helly number by one.

\begin{theorem}\label{theorem-hconvex}
	Let $H \subset S^{d-1}$ be a finite set of directions for which $H$-convex sets are closed under Minkowski sums.  Let $\mathcal{F}$ be a finite family of convex sets in $\rr^d$ and $\mu$ be a log-concave measure.  Suppose that the intersection of every $|H|$ or fewer sets in $\mathcal{F}$ contains an $H$-convex set of $\mu$ measure greater than or equal to one.  Then, $\cap \mathcal{F}$ contains an $H$-convex sets of $\mu$ measure greater than or equal to one.
\end{theorem}

\begin{proof}
	With the Minkowski parametrization in $\rr^{|H|}$ described above, we can consider
	\begin{align*}
	S(F)   = \{ & \left(\lambda_1, \ldots, \lambda_{|H|-1}\right) \in \rr^{|H|-1} : \mbox{ there exists $\lambda$ such that $(\lambda_1, \ldots, \lambda_{|H|-1}, \lambda)$ } \\ & \mbox{represents an $H$-convex set that has volume at least one and} \\ & \mbox{is contained in $F$} \}	
	\end{align*}
	
	An analogous argument to the one used in the proof of Theorem $\ref{theorem-zonotopes-strong}$ shows that $S(F)$ is convex, so the conclusion of the theorem follows.
	
\end{proof}

As an example, we can consider $H$ to be the set of directions of the form $\pm e_i$ where $e_i$ is the $i$-element of the canonical basis.  Then, $H$-convex sets are axis-parallel boxes.  We recover our Helly theorem for boxes with the same Helly number as before, $2d$.

 Zonotopes and $H$-convex sets are both families of polytopes for finite sets $H$.  The Helly number for theorems regarding $H$-convex sets is the number of possible facets such polytopes can have, while for zonotopes it's the number of directions in their $1$-skeleton.  In the plane, $H$-convex sets give us a much stronger result, since we can have polytopes which are not centrally symmetric.  Moreover, if in $\rr^d$ we have $H = -H$ and we only seek centrally symmetric $H$-convex sets, we can parametrize any $H$-convex set using $d+|H|/2$ parameters, thereby reducing the Helly number.

%The argument we showed for simultaneous $\varepsilon$-approximation by zonotopes extends to $H$-convex sets.  We simply state the main theorem.

%\begin{theorem}
%	Let $\mathcal{F}$ be a family of $n$ convex sets in $\rr^d$, and $H \subset S^{d-1}$ be a set such that $H$-convex sets are closed under Minkowski sums.  Assume that $H, d$ are fixed.  Then, the problem of determining the smallest $\varepsilon$ such that there exists an $H$-convex set that is a simultaneous $\varepsilon$-approximation for $\mathcal{F}$ can be solved in time $O(n^{|H|})$.
%\end{theorem}

\section{Helly results with a topological proof}\label{section-topological}

Given a finite family $\mathcal{F}$ of sets, we can define $N(\mathcal{F})$, the \textit{nerve complex} of $\mathcal{F}$, as a simplicial complex with one vertex for each element of $\mathcal{F}$, and include a face if the corresponding vertices represent a subfamily with a non-empty intersection.  A large family of variations of Helly's theorem relies on studying the nerve complex of a family of sets.  Understanding the topological properties of the sets in question and their nerve complexes is often all that is needed to prove Helly-type theorems \cite{Tancer:2013iza, ColindeVerdiere:2014gwa}.

Kalai and Meshulam proved a broad generalization of the colorful Helly theorem for simplicial complexes which are $d$-Leray.  We say a simplicial complex $X$ is $d$-Leray if the $i$-th reduced homology group $\tilde{H}_i (Y)$ over $\mathbb{Q}$ vanishes for every induced subcomplex $Y $ of $ X$ and every $i \ge d$.

\begin{theoremp}[Kalai, Meshulam 2005 \cite{Kalai:2005tb}]
	Let $X$ be a $d$-Leray complex on a set $V$ of vertices.  Let $M$ be a matroidal complex on the same set $V$ of vertices with a rank function $\rho$.  If $M \subset X$ then there exists a simplex $\tau \in X$ such that $\rho (V \setminus \tau) \le d$.
\end{theoremp}

In most of our applications, we only use the case when $M$ is a partition matroid.  In this matroid, we are given a partition $V= V_1 \biguplus V_2 \biguplus \ldots \biguplus V_k$ and we declare that a set $L \subset U$ is independent in $M$ if and only if $|L \cap V_i| \le 1$.  This makes the rank function $\rho(S)$ to be the number of indices $i$ such that $S \cap V_i \neq \emptyset$.  The classical colorful variations of Helly appear if $M$ is a partition matroid with $k=d+1$.

In order to apply Kalai and Meshulam's Helly theorem, we need to construct nerve complexes corresponding to quantitative intersections and bound their Leray number.  If a topological space $X$ is $n$-dimensional, then the homology groups of its open subsets vanish starting from $\tilde{H}_{n+1} ( \cdot )$.  In some cases, we can apply the following simple lemma to improve our bounds.

\begin{lemma}\label{lemma-leray}
	If a topological space $X$ is $n$-dimensional and $\tilde{H}_n(X) = 0$, then $\tilde{H}_k (Y) = 0$ for all $k \ge n$ and all open subsets $Y \subset X$.
\end{lemma}

\begin{proof}
	Let $Y$ be an open subset of $X$.  We know that $\tilde{H}_i (Y) = 0$ for $i > n$.  Assume that $\tilde{H}_n(Y) \neq 0$.  This means that there exists a non-zero element $[q] \in \tilde{H}_n (Y)$.  Consider a cycle in $[q]$ as a subset of $X$.  Since $X$ is $n$-dimensional, $q$ is not the boundary of an $(n+1)$-dimensional chain.  Therefore, it would be a non-zero element of $\tilde{H}_n(X)$, a contradiction.  We obtain $\tilde{H}_i (Y) = 0$ for $i \ge n$, as we wanted.
\end{proof}

The next ingredient we need is the classical nerve lemma, attributed to Borsuk and Leray \cite{Borsuk:1948tz, Leray:1950wn}.

\begin{lemma}[Nerve lemma]
	Let $\mathcal{F}$ be a finite collection of open subsets in a paracompact topological space $X$.  If every non-empty intersection of sets in $\mathcal{F}$ is contractible, then $N(\mathcal{F})$ is homotopy equivalent to $\cup \mathcal{F}$.
\end{lemma}

\subsection{Affine images revisited}\label{subsection-affine-images-revisited}
	As we saw in section \ref{example-matrices}, the affine images of a set $K \subset \rr^d$ given by symmetric positive-definite $d\times d$ matrices had a nice parametrization.  Consider
	\[
	\mathcal{K} = \{(a, A) : a \in \rr^d, \ A \mbox{ is a symmetric positive definite $d \times d$ matrix and } \det A = 1\}.
	\]
	Notice that the dimension of this space is $\frac{d(d+3)}2-1$.  Given $K \subset \rr^d$ an open set of volume one and $M \subset \rr^d$ a convex set, we define
	
	\[
	S_K(M, \vol =1) = \{(a,A) \in \mathcal{K} : a + AK \subset M\}.
	\]

\begin{lemma}\label{lemma-contractible1}
For $\mathcal{K}$, $K$, and $M$ defined as above, the set $S_K(M, \vol =1)$ is either empty or contractible.
\end{lemma}

\begin{proof}
	Let $(a,A) \in S_K(M, \vol = 1)$ be fixed.  We are going to explicitly give a strong deformation retract of $S_K(M, \vol = 1)$ to $\{(a,A)\}$.  Let $(b,B)$ be any other element of $S_K(M, \vol = 1)$, and $\lambda \in [0,1]$.  We define $f_{\lambda}(b,B) \in \mathcal{K}$ as
	\begin{align*}
		C_{\lambda} & = \lambda A + (1-\lambda)B \\
		c_\lambda & = \lambda a + (1-\lambda)b \\
		f_{\lambda} (b,B) & = \left(c_{\lambda}, \frac{1}{(\det C_{\lambda})^{1/d}}C_{\lambda}\right).
	\end{align*}
	
	We know from Section \ref{example-matrices} that $c_{\lambda} + C_{\lambda}K \subset M$.  Since the determinant is log-concave in the space of symmetric positive-definite matrices, we have that
	\[
	\det C_{\lambda} \ge \det(A)^{\lambda}\det(B)^{1-\lambda} = 1.
	\]
	Therefore, $c_{\lambda} + \left( \frac{1}{(\det C_{\lambda})^{1/d}}C_{\lambda}\right) K \subset c_\lambda + C_\lambda K \subset M$.  In other words, $f_\lambda (b, B) \in S_K(M, \vol =1 )$.  Notice that this function is continuous on $(b,B)$ and $\lambda$, that it is equal to $(b,B)$ if $\lambda = 0$ and equal to $(a,A)$ if $\lambda = 1$.  Therefore, it is the retract we wanted.
\end{proof}

\begin{corollary}\label{corollary-leray}
	Let $n = \frac{d(d+3)}2-1$.  Then, $\tilde{H}_n(\mathcal{K})=0$.
\end{corollary}

\begin{proof}
	The space $\mathcal{K}$ is $n$-dimensional.  Take any set $K \subset \rr^d$ of volume one.  Then, $\mathcal{K} = S_K(\rr^d, \vol = 1)$, which is contractible by Lemma \ref{lemma-contractible1}.  Therefore, $\mathcal{K}$ has trivial homology.  In particular, $\tilde{H}_n(\mathcal{K}) = 0$. 
\end{proof}

Now we are ready to prove a slightly stronger version of Theorem \ref{theorem-colorfulhellyellipsoids}.  We denote the space of $d \times d$ positive definite symmetric matrices by $\mathcal{P}_d$.

\begin{theorem}\label{theorem-strong-matroid-affine}
	Let $M$ be a matroid on a set $V$ of vertices with rank function $\rho$, and let $K \subset \rr^d$ be an open set of volume one.  For each $v$ in $V$, we are given an open convex set $F_v$ in $\rr^d$.  We know that for each set $V' \subset V$ that is independent in $M$, there exist $a \in \rr^d$ and $A \in \mathcal{P}_{d}$ such that $\det A = 1$ and $a + AK \subset \cap_{v \in V'} F_v$.  Then, there exists a set $\tau \subset V$ such that $\rho (V \setminus \tau) \le d(d+3)/2-1$ and for which there exist $a \in \rr^d$ and $A \in \mathcal{P}_{d}$ such that $\det A = 1$ and $a+AK \subset \cap_{v \in \tau} F_v$.
\end{theorem}

The theorem above has some additional flexibility.  For example, the determinant can be changed to other $\min$-concave functions $f$ on $\mathcal{P}_d$.  The only additional condition we require is that $f$ is continuous and $f(A) \ge f(\alpha A)$ for $0 < \alpha < 1$, so the shrinking argument works to prove contractibility.  We describe the consequences of using the determinant (which is log-concave) and the trace.  These two functions have a precise geometric meaning.  The determinant of $A$ is proportional to the volume of $a + AK$.  If $K$ is a ball, then the trace of $A$ is proportional to the sum of the lengths of the axes of the ellipsoid $a+AK$.  However, there is a vast number of known concave functions to choose from \cite{Lieb:1973eu, Ando:1979ju, Carlen:2010Ib, Hiai:2013kz}, which may lead to interesting Helly-type theorems.

\begin{proof}
  Let $n = d(d+3)/2-1$.  For each $v \in V$, we consider the set
  	\[
	G_v = S_K(F_v, \vol = 1) \subset \mathcal{K}.
	\]
	
	Let $X = N(\{G_v : v \in V\})$, which can be considered as a simplicial complex on $V$.  If $V' \subset V$ is a face of $X$, it means that for $\mathcal{G}' = \{G_v : v \in V'\}$ we have $\cap \mathcal{G}' \neq \emptyset$.  This allows us to describe the intersection as
	
	\[
	\cap \mathcal{G}' = \bigcap_{v \in V'} G_v = \bigcap_{v \in V'} S_K(F_v, \vol =1) = S_K\left(\left(\bigcap_{v \in V'} F_v\right), \vol = 1\right).
	\]

	% $\mathcal{F}_1, \ldots, \mathcal{F}_n$ be families of convex sets in $\rr^d$ satisfying the conditions of Theorem \ref{theorem-colorfulhellyellipsoids}.  Let $\mathcal{F} = \cup_{i=1}^n \mathcal{F}_i$, where sets are counted with multiplicity.  Let $B_d$ be a ball of unit volume in $\rr^d$.  We define a family of sets in $\mathcal{K}$ as
	%\begin{align*}
	%\mathcal{G}_i &  = \{ S_{B_d}(F, \vol = 1) : F \in \mathcal{F}_i\} \\
		%\mathcal{G} & = \cup_{i=1}^n \mathcal{G}_i.	
	%\end{align*}
	%We consider $\mathcal{G}$ as a multiset if the families $\mathcal{G}_i$ share elements.  Consider the simplicial complex $N(\mathcal{G})$.  For each subfamily $\mathcal{G}' \subset \mathcal{G}$ such that $\cap \mathcal{G}' \neq \emptyset$, let $\mathcal{F}' \subset \mathcal{F}$ be the family of sets such that
	%\[
	%G \in \mathcal{G}' \Leftrightarrow G = S_{B_d}(F, \vol = 1) \mbox{ for some }F \in \mathcal{F}'.
	%\] 
	%This allows us to describe the intersection of $\mathcal{G}'$ as follows.

	Therefore, $\cap \mathcal{G}'$ is contractible.  This means we can apply the nerve lemma, and $N(\mathcal{G})$ is homotopy equivalent to $\cup{\mathcal{G}}$.  Moreover, the complex induced by $V'$ is homotopy equivalent to $\cup \mathcal{G}'$, whose reduced homology groups vanish starting from $\tilde{H}_n(\cdot)$.  Therefore, $N(\mathcal{G})$ is $n$-Leray.  Now we can apply the Kalai-Meshulam theorem to $N(\mathcal{G})$, giving us the desired result.
	
	\end{proof}
	
	The fact that Corollary \ref{corollary-colorfulhellyellipsoids} is optimal also shows that theorems \ref{theorem-colorfulhellyellipsoids} and \ref{theorem-strong-matroid-affine} are optimal.  Let us prove Corollary \ref{corollary-colorfulhellyellipsoids}.
	
	\begin{proof}[Proof of Corollary \ref{corollary-colorfulhellyellipsoids}]
	
	\textbf{Upper bound.}
	Notice that if $B_d$ is a ball of volume $1$, then the set $a + AB_d$ where $a \in \rr^d$, $A \in \mathcal{P}_d$ parametrizes all ellipsoids in $\rr^d$.  The volume of the ellipsoid $a + AB_d$ is precisely $\det (A)$.
	Let $\mathcal{F} = \mathcal{F}_1 \cup \ldots \cup \mathcal{F}_{n}$, where sets are counted with multiplicity.  The partition induced by the $\mathcal{F}_i$ creates a matroid structure on $\mathcal{F}$.  We can apply Theorem \ref{theorem-colorfulhellyellipsoids} and obtain the upper bound of Corollary \ref{corollary-colorfulhellyellipsoids}.

	\textbf{Lower bound.}  We show how to construct a family $\mathcal{F}$ of $\frac{d(d+3)}{2}$ convex sets such that $\cap \mathcal{F}$ does not contain an ellipsoid of volume greater than $1$, but the intersection of any $\frac{d(d+3)}{2}-1$ or fewer sets of $\mathcal{F}$ does contain an ellipsoid of volume strictly greater than $1$.  The construction we made turned out to be the same as Dam\'asdi's \cite{Damasdi:2017ta}, but we include it for completeness.   If we take $\mathcal{F}_1 = \ldots = \mathcal{F}_n = \mathcal{F}$ and scale everything appropriately, we have the desired counter-example.
	
	It is known that for most convex sets, the number of contact points with its John ellipsoid is precisely $\frac{d(d+3)}{2}$ \cite{Gruber:1988gn, Gruber:2011dx}.  Let $K$ be such a convex body.  By applying an appropriate affine transformation, we can assume that the John ellipsoid of $K$ is the unit ball $\tilde{B}_d \subset \rr^d$.  Let $n = d(d+3)/2$ as before and consider $u_1, \ldots, u_n$ the contact points of $K$ with $\tilde{B}_d$.  The classical characterization of sets whose John ellipsoid is the unit ball is that there are non-negative coefficients $\lambda_1, \ldots, \lambda_n$ such that
	\begin{align*}
		\sum_{i=1}^n \lambda_i (u_i \otimes u_i )&= I_{d \times d} \\
		\sum_{i=1}^n \lambda_i u_i & = 0
	\end{align*}
	The set of matrices of the form $(u, 1) \otimes u$, where $\otimes$ denotes the tensor product, lies in an $n$-dimensional affine space of the space of $(d+1) \times d$ matrices.  If we also restrict $u$ to be a unit vector, this makes the trace of $u \otimes u$ to be equal to $1$, so $(u,1) \otimes u$ is in an $(n-1)$-dimensional affine space.  Moreover, the trace shows that $\sum \lambda_i = d$.  Therefore, we can modify the expression above to 
	\begin{align*}
		\sum_{i=1}^n \left( \frac{\lambda_i}{d}\right) (u_i,1) \otimes u_i = \frac1d J.
	\end{align*}
The matrix $J$ is a $(d+1)\times d$ matrix formed by a $d\times d$ identity matrix with an extra row of zeros, and the expression above is a convex combination.  This is consistent with Carath\'eodory's theorem: $n$ elements are expected to be necessary to contain the point $(1/d)J$ in their convex hull if we choose them from an $(n-1)$-dimensional space.  What Gruber's results show is that this often optimal: for most convex sets the $n$-tuple $\{(u_i, 1) \otimes u_i : i =1,\ldots, n\}$ is critical, as none of its proper subsets contains $(1/d)J$ in its convex hull.

Now, given this $n$-tuple of contact points $u_1, \ldots, u_n$, consider $\mathcal{F}$ the family of halfspaces of the form $\{ x : \langle x, u_i\rangle \le 1\}$ for each $i$.  By the characterization of the John ellipsoid, the unit ball is the maximal ellipsoid in $\cap \mathcal{F}$.  However, for any proper subset $\mathcal{F}' \subset \mathcal{F}$, we have that $\cap \mathcal{F}'$ has fewer contact points with the unit ball centered at the origin.  For those contact points $u$, the convex hull of the points $(u,1) \otimes u$ cannot contain $(1/d)J$, so $B_d$ is not the maximal ellipsoid of this set.  Since $B_d \subset \cap \mathcal{F}'$, we have that $\mathcal{F}'$ must contain an ellipsoid of strictly larger volume.
\end{proof}

Now we are ready to prove Theorem \ref{theorem-truecolorful-volumentric}.

\begin{proof}[Proof of Theorem \ref{theorem-truecolorful-volumentric}]
	We are going to use Brazitikos' volumetric Helly theorem \cite{Brazitikos:2017ts} for this proof.  If one goes through the proof of Brazitikos' result, he actually proves that \textit{given a finite family $\mathcal{F}$, if the intersection of every $2d$ of its sets has volume at least one, then $\cap \mathcal{F}$ contains an ellipsoid of volume $d^{-3d/2+o(1)}$}.  We are going to use Brazitikos' ellipsoid to our advantage.
	
	If we take an independent set $V'$, the condition of the theorem implies that the intersection of every $2d$ of the convex sets represented by the vertices in $V'$ has volume greater than or equal to one.  Therefore, $\cap_{v \in V'} F_v$ contains an ellipsoid of volume $d^{-3d/2+o(1)}$.  Now we can apply Theorem \ref{theorem-colorfulhellyellipsoids} and conclude that for some set $\tau$ with $\rho(V \setminus \tau)\le d(d+3)/2-1$, we have that $\cap_{v \in \tau} F_v$ also contains an ellipsoid of volume $d^{-3d/2+o(1)}$.
\end{proof}

\subsection{Further variations and diameter results}\label{secion-more-ellipsoids}

In this subsection, we discuss how we can modify the setting used in the topological proofs to obtain other variations of Helly's theorem.   We only describe standard Helly theorems, although every theorem has a general matroid version.  Most of these require minimal modifications to the arguments presented earlier.  Recall that $\mathcal{P}_d$ stands for the set of symmetric positive definite matrices.  In this subsection, the set $\mathcal{F}$ will always be a finite family of convex sets in $\rr^d$.

\begin{example}[Translates of a convex set]\label{example-translates}
	Instead of pairs $(a, A)$ in $\rr^d \times \mathcal{P}_d$, consider only the pairs of the form $(a,I)$, where $I$ the identity matrix.  The dimension of this space is equal to $d$.  Let $K$ be a fixed set in $\rr^d$.  We consider again
	\begin{align*}
		S_K(M)=\{(a,I): & a \in \rr^d, a + I K \subset M \}.
	\end{align*}
	  The arguments above re-prove the folklore theorem \textit{``If the intersection of every $d+1$ or fewer sets of $\mathcal{F}$ contains a translate of $K$, then $\cap \mathcal{F}$ contains a translate of $K$''}.  Of course, the identity matrix is superfluous in this case, but it shows how our methods relate to this classical result.
\end{example}

The example above has the particular advantage that $S_K(M)$ is easy to visualize.  If $M+T+K$ for some convex set $T$, then $S_K(M) = T$.  

\begin{example}[Axis-parallel ellipsoids]\label{example-diagonal}
	Instead of all pairs $(a, A)$ in $\rr^d \times \mathcal{P}_d$, consider only the pairs where $A$ is a  positive definite diagonal matrix.  The resulting space now has dimension $2d$.  We can apply the topological Helly theorem to the family of sets
	\begin{align*}
		S(M)=\{(a,A): & a \in \rr^d, a + AB_d \subset M, A \mbox{ is a diagonal matrix}, \det(A) = 1 \}.
	\end{align*}   This shows that ``If the intersection of every $2d$ or fewer sets of $\mathcal{F}$ contains an ellipsoid of volume one with axes parallel to the canonical basis, so does $\cap \mathcal{F}$''.
\end{example}

We can modify this example to prove the following diameter Helly theorem version with boxes.

%Another interesting instance of Example \ref{example-diagonal} is when $B_d$ is replaced by an axis-parallel box and the determinant is replaced by the trance.  We quickly obtain quantitative Helly theorem for axis-parallel boxes whose sum of side-lengths is one.  We can 

\begin{theorem}[Helly for box diameter]\label{theorem-boxdiameter}
	Let $\mathcal{F}$ be a finite family of convex sets in $\rr^d$ such that the intersection of every $2d$ or contains an axis-parallel box of diameter one.  Then, $\cap \mathcal{F}$ contains an axis-parallel box of diameter $d^{-1/2}$.
\end{theorem}

\begin{proof}
	In Example \ref{example-diagonal}, replace $B_d$ by an axis-parallel hypercube of side-length $\frac{1}{d}$, and the determinant for the trace.  The intersection of any $2d$ or fewer sets of our family contains an axis-parallel box whose diameter is at least one.  For this box the sum of its side-lengths must also be at least one.  Therefore, by our modification of Example \ref{example-diagonal}, $\cap \mathcal{F}$ must contain an axis-parallel box whose side-lengths add up to one.  A simple application of the arithmetic mean - quadratic mean inequality shows that the diameter of this box is at least $d^{-1/2}$.
\end{proof}

We should note that $O(d^{-1/2})$ is the B\'ar\'any-Katchalski-Pach conjecture for Helly's theorem for the diameter if we know that the intersection of every $2d$ sets has diameter greater than or equal to one \cite{Barany:1982ga}.  Brazitikos has confirmed the conjecture for families of centrally symmetric sets \cite{Brazitikos:2016ja}.  We just confirmed their conjecture for diameters realized by axis-parallel boxes.  We can also confirm it for the ``increasing'' diameter.  Given two vectors $x=(x_1, \ldots, x_d), y=(y_1, \ldots, y_d)$, we say that $x \ge y$ if $x_i \ge y_i$ for all $i\in\{1,2,\ldots, d\}$.  Given a compact set $K \subset \rr^d$, we define its increasing diameter as $\max \{||x-y||: x \ge y, x \in K, y\in K\}$.  The norm used in this definition is the $\ell_2$-norm.

\begin{theorem}[Helly for increasing diameter]\label{theorem-increasingdiameter}
	Let $\mathcal{F}$ be a finite family of compact convex sets in $\rr^d$.  If the intersection of every $2d$ of them has an increasing diameter greater than or equal to one, then the increasing diameter of $\cap \mathcal{F}$ is greater than or equal to $d^{-1/2}$.
\end{theorem}

\begin{proof}
	We modify Example \ref{example-diagonal} again, but now we replace $B_d$ by a segment in the direction $(1,\ldots, 1)$ of unit length in the $1$-norm, and the determinant by the trace.  For any $2d$ or fewer sets of $\mathcal{F}$, their intersection contains an increasing interval of $2$-norm equal to $1$.  Since the $1$-norm of this interval must be at least one, we can apply the resulting Helly theorem from the example and obtain an increasing segment in $\cap \mathcal{F}$ of $1$-norm equal to $1$.  Finally, the $2$-norm of this interval is at least $d^{-1/2}$, as we wanted to prove.
\end{proof}

Theorems \ref{theorem-boxdiameter} and \ref{theorem-increasingdiameter} are similar, yet neither seems to directly imply the other.  They do suggest that the issues for the quantitative theorems for the diameter may be due to the norm selected.

Whenever we have a parametrization of a family of convex sets $D : \rr^l \to \mathcal{C}$ such that $D(x + y) = D(x) \oplus D(y)$, we also get a diameter Helly theorem.  Let us give as an example the result for $H$-convex sets.

\begin{theorem}[Diameter Helly for $H$-convex sets]\label{theorem-diameter-xue}
	Let $H \subset S^{d-1}$ be a finite set of directions for which $H$-convex sets are closed under Minkowski sum.  Let $\mathcal{F}$ be a finite family of convex sets in $\rr^d$.  If the intersection of every $|H|$ or fewer sets of $\mathcal{F}$ contains an $H$-convex set of diameter greater than or equal to one, then $\cap \mathcal{F}$ contains an $H$-convex set of diameter greater than or equal to $|H|^{-1/2}$.
\end{theorem}

\begin{proof}
	Given $K \in \mathcal{F}$, let 
	\begin{align*}
		S(K) = \{ & (\lambda_1, \ldots, \lambda_{|H|-1}) \in \rr^{|H|-1}: \mbox{ there exists $\lambda$ such that $(\lambda_1, \ldots, \lambda_{|H|-1},\lambda)$} \\ &\mbox{ represents an $H$-convex set contained in $K$ of diameter at least one}\}
	\end{align*}
	
	Since the diameter is not a concave function, $S(K)$ may not be convex.  However, consider the family $\mathcal{G} = \{ \conv (S(K)) : K \in \mathcal{F}\}$.  If we apply Helly's theorem to $\mathcal{G}$, we obtain a point $p$ in its intersection.  Let us show that, given a $K \in \mathcal{F}$, there exists a value of $\lambda$ such that $(p, \lambda)$ represents an $H$-convex set of diameter at least $|H|^{-1/2}$ contained in $K$.  The minimum value of $\lambda$ we obtain among all possible $K \in \mathcal{F}$ will give us the representation of the $H$-convex set we are looking for.
	
	Take any set $K \in \mathcal{F}$.  Since $p \in \cap \mathcal{G}$, we know that $p \in \conv(S(K))$.  Therefore, by Carath\'eodory's theorem we have that $p$ is the convex combination of $|H|$ points $p_1, \ldots, p_{|H|}$ of $S(K)$, $p = \sum_{i=1}^{|H|} \alpha_i p_i$.  For each $p_i$ there exists a value $\gamma_i$ such that $(p_i, \gamma_i)$ represents an $H$-convex set $K_i \subset K$ of diameter greater than or equal to one.  Therefore, $ \sum_{i=1}^{|H|} \alpha_i (p_i, \gamma_i)$ represents an $H$-convex set $K^* \subset K$.
	
	Notice that if $A, B$ are two sets in $\rr^d$, then 
	\[
	\diam (A \oplus B)^2 \ge \diam (A)^2 + \diam (B)^2.
	\]
	The inequality above follows easily for parallelograms.  Since the diameter of a set is realized by a segment, this shows that it holds for all sets.  Therefore,
	\begin{align*}
	\diam (K^*)^2 = \diam \left(\alpha_1 K_1 \oplus \ldots \oplus \alpha_{|H|} K_{|H|}\right)^2 \ge \sum_{i=1}^{|H|}\diam( \alpha_i K_i)^2 \ge \sum_{i=1}^{|H|}\alpha_i^2 \ge |H|^{-1}.
	\end{align*}
\end{proof}

The proof above also works with zonotopes, and the guarantee for diameter we obtain in the end is $(k+d)^{-1/2}$.  In the case of axis-parallel boxes, Theorem \ref{theorem-diameter-xue} gives a slightly weaker bound than Theorem \ref{theorem-boxdiameter}.  However, it is only off by a constant factor and applies to a much more general family of sets.

\begin{example}[Central symmetry]
	Suppose every set in $\mathcal{F}$ is centrally symmetric around the origin.  Then, instead of pairs $(a, A)$ in $\rr^d \times \mathcal{P}_d$ we only need to consider pairs of the form $(0, A)$.  This reduces the dimension of the set of pairs $(0,A)$ to $d(d+1)/2$.   Therefore, if we only consider centrally symmetric sets in Theorem \ref{theorem-colorfulhellyellipsoids}, the Helly number is reduced from $d(d+3)/2$ to $d(d+1)/2$.
\end{example}

\begin{example}[Flipping the containment]
	The equation $M \subset a + AB_d$ is equivalent to $-A^{-1}a + A^{-1} M \subset B_d$.  We know that the set of all pairs $(b,B) \in \rr^d \times \mathcal{P}_d$ such that $b +B M \subset B_d$ is convex.  Moreover, if $\det(A) = 1$, then $\det(A^{-1})=1$. Therefore, in the space $\rr^d \times \mathcal{P}_d$ the diffeomorphism $(a,A) \mapsto (-A^{-1}a, A^{-1})$ shows that the set
	\[
	\{(a,A) : M \subset a + A B_d, \det(A) = 1\}
	\]
	is contractible for each bounded set $M \subset \rr^d$.
\end{example}

The construction in the example above can be used as before to prove a version of Theorem \ref{theorem-colorfulhellyellipsoids} and Corollary \ref{corollary-colorfulhellyellipsoids} for enclosing ellipsoids.  The proof is identical to the one in the previous subsection with essentially the same proof.  We state below the flipped version of Corollary \ref{corollary-colorfulhellyellipsoids}.

\begin{theorem}\label{theorem-loewner-ellipsoids}
	Let $n = d(d+3)/2$ and let $\mathcal{F}_1, \ldots, \mathcal{F}_n$ be finite families of bounded sets in $\rr^d$.  Suppose that for every choice $F_1 \in \mathcal{F}_1, \ldots , F_n \in \mathcal{F}_n$ we have that $\cup_{i=1}^n F_i$ is contained in an ellipsoid of volume one.  Then, there exists an index $i$ such that $\cup \mathcal{F}_i$ is contained in an ellipsoid of volume one.
\end{theorem}

\subsection{Algorithmic consequences}

Most of our Helly-type results can be used to produce efficient algorithms to solve optimization problems.  Helly-type results and their connection to LP-type problems (short for ``linear programming'' type problems) have been used to design algorithms for a vast array of optimization problems.  Sharir and Welzl described a short list of conditions an LP-type problem needs in order to have a randomized algorithm that solves it in expected linear time in terms of the number of constraints \cite{Sharir:1992ih}.  Currently, there is a clear hierarchy of LP-type problems and the algorithms to solve them have been improved \cite{Gartner:2008bp, Amenta:2017ed}.

In addition to applying their framework to solve LP-problems, Sharir and Welzl showed how it could be used to compute the largest volume ellipsoid contained in the intersection of a family of $n$ half-spaces in $\rr^d$, or the smallest volume ellipsoid containing a family of $n$ points in $\rr^d$.  Our methods prove that the same algorithms can be used for a much wider family of functions and witness sets.  For example, consider the following two optimization results.

\begin{theorem}[Largest-volume axis-parallel box in an intersection]
	Suppose that we have an oracle that can determine the largest volume axis-parallel box in the intersection of any $2d$ convex sets in $\rr^d$.  Then, given a family $\mathcal{F}$ of convex sets in $\rr^d$, there is a randomized algorithm that finds the largest volume axis-parallel box in $\cap \mathcal{F}$ in expected $O(|\mathcal{F}|)$ calls to the oracle.
\end{theorem}

\begin{theorem}[Best simultaneous approximation to a family of convex sets]\label{theorem-algorithm-approx}
	Suppose $\mathcal{C}$ is a family of convex sets in $\rr^d$ with a Minkowski parametrization in $\rr^l$.  Let $\mathcal{F}$ be a finite family of convex sets in $\rr^d$.  Suppose there is an oracle that can compute, for any $l+d+1$ convex sets, the smallest $\varepsilon>0$ such that there exists an element $M$ of $\mathcal{C}$ with a translation that is a simultaneous $\varepsilon$-approximation for the $l+d+1$ sets.  Then, there is a randomized algorithm that computes the smallest $\varepsilon >0$ such that there exists an element $M$ of $\mathcal{C}$ with a translation that is a simultaneous $\varepsilon$-approximation for $\mathcal{F}$ in expected $O(|\mathcal{F}|)$ calls to the oracle.
\end{theorem}

The randomized dual-simplex algorithm of Sharir and Welzl is sufficient to prove the two theorems above \cite{Sharir:1992ih}. The algorithm for Theorem \ref{theorem-largest-box-algorithm} would look as follows:\\[0.5pt]

$\bullet$ \texttt{Order the sets in $\mathcal{F}$ randomly as $K_1, \ldots, K_{|\mathcal{F}|}$.}

$\bullet$ \texttt{Find the largest volume axis-parallel box $B$ in the intersection of $K_1, \ldots, K_{2d}$.}

$\bullet$ \texttt{Check one by one each $K_i$ to see if $B \subset K_i$.  If this hold for all, we are done and $B$ is the box we wanted.}

$\bullet$  \texttt{Otherwise, let $m$ be the smallest index for which $B \not\subset K_m$.  Find the largest volume axis-parallel box $B'$ in the intersection of $K_1, \ldots, K_{2d}, K_m$ by checking every $2d$-tuple.  Let $\mathcal{B}$ be the $2d$-tuple whose maximal volume axis-parallel box is $B'$.}

$\bullet$ \texttt{Run the algorithm recursively on $K_1, \ldots, K_m$ to find the largest volume axis-parallel box $B''$ in their intersection, but fixing the first $2d$ elements to be $\mathcal{B}$ in the initial random ordering.  This will also output the $2d$-tuple $\mathcal{B}'$ whose maximal volume axis-parallel box is $B''$.}

$\bullet$ \texttt{Continue with $K_{m+1}, \ldots, K_{|\mathcal{F}|}$.}

\vspace{5pt}

The validity of our Helly theorems implies the expected running time of the algorithms.  Our Helly theorems bound the combinatorial dimension of the associated optimization problems ($2d$ and $l+d+1$, respectively).  Of course, most of our Helly-type theorems yield an analogous linear-time algorithm for their associated optimization problem.  We stress that we have only described the dependence of the algorithms in terms of $|\mathcal{F}|$.   For more recent algorithms applicable to LP-type problems and their dependence on the combinatorial dimension of the problem, we recommend the discussion of the subject by G\"artner et al. and by Amenta et al. \cite{Amenta:2017ed, Gartner:2008bp}.

\section{Quantitative Tverberg results}\label{section-tverberg}

\subsection{Results using Tverberg's theorem}\label{section-affinetverberg}

In order to obtain a general Tverberg theorem, similar in spirit to Theorem \ref{theorem-general-helly}, we require slightly different conditions.  Suppose that $\mathcal{C}$ is a family of sets in $\rr^d$ parametrized by points in $\rr^l$, by a some function $D: \rr^l \to \mathcal{C}$.  We need that for every $F \subset \rr^l$, 
\[
D(\conv{F}) \subset \conv D(F).
\]

If $D$ is a Minkowski parametrization, the condition above comes for free.  To observe this, notice that for any two sets $A, B$ in $\rr^d$ we have that
\[
\bigcup_{\lambda \in [0,1]} (\lambda A \oplus (1-\lambda) B) \subset \conv(A \cup B).
\]

\begin{theorem}\label{theorem-general-Tverberg}
	Let $\mathcal{C}$ be a family of convex sets in $\rr^d$, and $D: \rr^l \to \mathcal{C}$ be a surjective function.  Suppose that for every set $\mathcal{F} \subset \rr^l$ we have $D(\conv{F}) \subset (\conv D(F))$.  Let $f: \mathcal{C} \to \rr$ be such that $f \circ D: \rr^l \to \rr$ is $\min$-concave.  Then, the following statement is true.
	 
	Given a family $\mathcal{F} \subset \mathcal{C}$ of cardinality $(r-1)(l+1)+1$ such that $f(K) \ge 1$ for all $K \in \mathcal{F}$, there exists a partition of $\mathcal{F}$ into $r$ sets $\mathcal{A}_1, \ldots, \mathcal{A}_r$ such that
	\[
	f\left(\bigcap_{j=1}^r \conv(\cup \mathcal{A}_j)  \right) \ge 1.
	\]
\end{theorem}

\begin{proof}
	For each $K \in \mathcal{F}$, consider a point $k \in \rr^l$ such that $D(k) = K$.  Then, applying Tverberg's theorem to the family of points in $\rr^l$ generated this way gives us the desired result. 
\end{proof}

We can apply almost any variation of Tverberg's theorem to Theorem \ref{theorem-general-Tverberg}, including the colorful Tverberg theorem or the versions with tolerance \cite{Barany:2018fy, DeLoera:2019jb}.  Let us describe some examples of families of convex sets for which Theorem \ref{theorem-general-Tverberg} applies. If we apply Theorem \ref{theorem-general-Tverberg} to the parametrization of ellipsoids described in Section \ref{example-matrices}, and $f( \cdot )$ is the volume, we obtain the following theorem, which directly implies Theorem \ref{theorem-tverbergexample}.

\begin{theorem}[Tverberg for ellipsoids of volume one]
	Let $\mathcal{F}$ be a family of $(r-1)\left(\frac{d(d+3)}{2}+1\right) + 1$ ellipsoids of volume one in $\rr^d$.  Then, there exists a partition of $\mathcal{F}$ into $r$ parts $\mathcal{A}_1, \ldots, \mathcal{A}_r$ such that $\bigcap_{j=1}^r \conv \left( \cup \mathcal{A}_j\right)$ contains an ellipsoid of volume one.
\end{theorem}

  The dimension of the space in Theorem \ref{theorem-general-Tverberg} can be reduced for zonotopes and $H$-convex sets exactly as in the proofs of Theorem \ref{theorem-zonotopes-strong} and Theorem \ref{theorem-hconvex}, by hiding one coordinate.  For zonotopes, we obtain the following result.

\begin{theorem}[Tverberg for zonotopes]
	Let $k \ge d$ be positive integers and ${v}_1, \ldots, v_k$ be directions in $\rr^d$.  Given a family $\mathcal{F}$ of $(r-1)(k+d)+1$ zonotopes with directions $v_1, \ldots, v_k$, each of volume at least one, there exists a partition of $\mathcal{F}$ into $r$ parts $\mathcal{A}_1, \ldots, \mathcal{A}_r$ such that $\cap_{j=1}^r (\conv(\cup \mathcal{A}_j))$ contains a zonotopes with directions $v_1, \ldots, v_k$ of volume at least one.
\end{theorem}

For the case $k=d$, in which our zonotopes are axis-parallel boxes, the results above requires $2(r-1)d+1$ boxes, which is linear in both $d$ and $r$.  The previous techniques used to obtain quantitative Tverberg theorems \cite{Soberon:2016co, Rolnick:2017cm, DeLoera:2017bl} give much larger dependence on the dimension.

In some cases, such as for ellipsoids centered at the origin, we can reduce the number of sets that Theorem \ref{theorem-general-Tverberg} requires.

\begin{theorem}[Tverberg for ellipsoids centered at the origin]\label{theorem-tverberg-centrallysymmetric}
	Let $r,d$ be positive integers.  Given a set $\mathcal{F}$ of $(r-1)\frac{d(d+1)}{2}+1$ ellipsoids of volume one in $\rr^d$ centered at the origin, there exists a partition of $\mathcal{F}$ into $r$ sets $A_1, \ldots, A_r$ such that
	\[
	\bigcap_{j=1}^r \conv \left( \cup A_j \right)
	\]
	contains an ellipsoid of volume one centered at the origin.
\end{theorem}

\begin{proof}
	We use a slightly different parametrization of ellipsoids than the one used in the proof of Theorem \ref{theorem-strong-matroid-affine}.  Let $\mathcal{S}_d$ be the space of symmetric positive definite matrices whose sum of entries is equal to one.  Notice that $S_d$ is an affine subspace of the cone $\mathcal{P}_d$ of symmetric positive definite matrices.  The set $S_d$ has dimension $\frac{d(d+1)}{2}-1$, so we can consider it as a real vector space.
	
	Given a convex set $M \subset \rr^d$, let us consider
	\[
	S^{*} (M) = \left\{A \in \mathcal{S}_d : \left( \frac{1}{\det(A)^{1/d}}A\right) B_d \subset M\right\}.
	\]
	Let us show that $S^{*}(M)$ is convex.  Notice that since the parametrization was modified b, so does the meaning of convex combinations.  Let $\lambda \in [0,1]$ and $A, B \in S^*(M)$.  Consider
	\begin{align*}
		A'  & = \frac{1}{\det(A)^{1/d}}A, \\
		B'  & = \frac{1}{\det(B)^{1/d}}B, \\ 
		s & = \det(A)^{1/d} \lambda, \qquad \mbox{and} \\
		t & = \det(B)^{1/d} (1-\lambda).
	\end{align*}
	
	We want to show that $ \lambda A + (1-\lambda)B \in S^*(M)$.  We can rewrite the matrix by noticing that $\lambda A + (1-\lambda)B = s A' + t B'$.  Therefore, we want to show that 
	\[
	\frac{1}{\det(sA'+tB')^{1/d}}(sA' + tB') B_d \subset M.
	\]
	
	We know that $A'B_d \subset M$ and $B'B_d \subset M$, so $\frac{1}{s+t}(sA'+tB')B_d \subset M$.  Now, it remains to use the log-concavity of the determinant to show the following inequality.
	\begin{align*}
	\frac{1}{\det(sA'+tB')^{1/d}} & = \frac{1}{(s+t)\det((\frac{s}{s+t})A'+(\frac{t}{s+t})B')^{1/d}} \\ & \le \frac{1}{(s+t)\det(A')^{s/d(s+t)}\det(B')^{t/d(s+t)}} = \frac{1}{s+t}.	
	\end{align*}
	Therefore, $\lambda A + (1-\lambda)B \in S^*(M)$, and the problem reduces to Tverberg's theorem on $\mathcal{S}_d$.
\end{proof}

If the ellipsoids are no longer centered at the origin, the argument above fails to show that $S^*(M)$ is convex.  In the next subsection we use the topological version of Tverberg's theorem to get around this problem.  The theorem above does, however, imply the following alternative Tverberg theorem for the volume.

\begin{theorem}[Tverberg for the volume of centrally symmetric sets]
	Let $r, d$ be positive integers.  If we are given a family $\mathcal{F}$ of $(r-1)\frac{d(d+1)}{2}+1$ centrally symmetric convex sets of volume one, each centered at the origin, there exists a partition of $\mathcal{F}$ into $r$ parts $\mathcal{A}_1, \ldots, \mathcal{A}_r$ such that the volume of $\bigcap_{j=1}^r \conv( \cup \mathcal{A}_j)$ is at least $d^{-d/2}$.
\end{theorem}

\begin{proof}
In order to prove the Tverberg theorem above, the only additional fact we need is that for a centrally symmetric convex set $K$ with John Ellipsoid $\mathcal{E}$, we have
\[
\mathcal{E} \subset K \subset \sqrt{d} \ \mathcal{E}.
\]

Once combined with Theorem \ref{theorem-tverberg-centrallysymmetric} we obtain the desired conclusion.
\end{proof}

An application of the same parametrization as in Theorem \ref{theorem-increasingdiameter} gives the following Tverberg-type result.  We say a segment is increasing if its endpoints $x,y$ satisfy $x \ge y$ or $y \ge x$, with the partial order considered in Section \ref{secion-more-ellipsoids}.

\begin{theorem}[Tverberg for increasing segments]
	Given a family of $2(r-1)d+1$ increasing segments in $\rr^d$, each with $\ell_1$-norm equal to one, there exists a partition of the family into $r$ parts  $\mathcal{A}_1, \ldots, \mathcal{A}_r$ such that $\bigcap_{j=1}^r \conv (\cup \mathcal{A}_j)$ contains an increasing segment with $\ell_1$-norm equal to one.
\end{theorem}

Since there are $2^{d-1}$ different diagonals in a hypercube, the result above gives an exact Tverberg theorem for $\ell_1$ norm.

\begin{corollary}[Exact Tverberg for $\ell_1$ diameter]
	Let $\mathcal{F}$ be a family of $(r-1)2^d d + 1$ segments in $\rr^d$, each with $\ell_1$-norm equal to one.  Then, there exists a partition of the family into $r$ parts $\mathcal{A}_1, \ldots, \mathcal{A}_r$ such that $\bigcap_{j=1}^r \conv (\cup \mathcal{A}_j)$ contains a segment of $\ell_1$-norm equal to one.
\end{corollary}

We do not know if the dependence of $d$ in the corollary above must be exponential in the dimension.  It does imply a version for the standard Euclidean norm, where the conclusion guarantees a segment of length $d^{-1/2}$.

\subsection{Results using the topological Tverberg theorem}

The topological Tverberg theorem states the following. 

\begin{theoremp}[Topological Tverberg \cite{Barany:1981vh, Oza87, Volovikov:1996up}]
Let $r$ be a prime power, $d$ be a positive integer, $N=(r-1)(d+1)$, and $\Delta_N$ be an $N$-dimensional simplex.  Then, for every continuous function $f:\Delta_N \to \rr^d$, there exist $r$ points $x_1, \ldots, x_r$ in pairwise vertex-disjoint faces of $\Delta_N$ such that $f(x_1) = \ldots = f(x_r)$.	
\end{theoremp}

Tverberg's theorem is the case for affine functions $f$, which does not need the conditions on $r$.  The requirement that $r$ is a prime power necessary \cite{Mabillard:2015vx, Frick:2015wp, Blagojevic:2017bl}.  As is usual with applications of the topological Tverberg theorem, it is not clear if the counterexamples when $r$ is not a prime power can arise from the functions we construct below.

\begin{theorem}\label{theorem-tverberg-topological-ellipsoids}
	Let $r$ be a prime power, and $d$ be a positive integer.  Then, given a family $\mathcal{F}$ of $(r-1)\frac{d(d+3)}{2}+1$ ellipsoids of volume one, there exists a partition of $\mathcal{F}$ into $r$ parts $\mathcal{A}_1, \ldots, \mathcal{A}_r$ such that $\bigcap_{j=1}^r \conv(\cup \mathcal{A}_j)$ contains an ellipsoid of volume one. 
\end{theorem}

\begin{proof}
	Notice that the set $\mathcal{Q}_d$ of symmetric positive definite matrices with determinant one and the set $\mathcal{S}_d$ of symmetric positive definite matrices with sum of entries equal to one are homeomorphic.  It suffices to consider
	\begin{align*}
		\mathcal{S}_d & \to \mathcal{Q}_d \qquad & \mathcal{Q}_d  & \to \mathcal{S}_d \\
		A & \mapsto \frac{1}{\det (A)^{1/d}} A \qquad & B & \mapsto \frac{1}{\operatorname{sum}(B)}B.
	\end{align*}
	
	In the expression above, $\operatorname{sum}(B)$ denotes the sum of the entries of $B$.  Since $\mathcal{S}_d$ can be identified with $\rr^{d(d+1)/2 - 1}$, so can $\mathcal{Q}_d$.
	
	Now, given $k$ points $(a_1, A_1), \ldots, (a_k, A_k)$ in $\rr^d \times \mathcal{Q}_d$ and $\lambda_1, \ldots, \lambda_k$ the coefficients of a convex combination, we can construct
	\begin{align*}
		c & = \sum_{i=1}^k \lambda_i a_i \\
		C^* & = \sum_{i=1}^k \lambda_i A_i \\
		C & = \frac{1}{\det(C^*)^{1/d}}C^*
	\end{align*}

If $M$ is a convex set and $a_i + A_i B_d \subset M$ for all $i$, then $c + C^* B_d \subset M$.  However, since the determinant is log-concave in the space of positive definite matrices, we have that
\[
\det (C^*) = \det \left(\sum_{i=1}^k \lambda_i A_i\right) \ge \prod_{i=1}^k \det (A_i)^{\lambda_i} = 1.
\]

Therefore, $c + C B_d \subset M$.  We consider the pair $(c,C)$ to represent the convex combination of the points $(a_i, A_i)$ with coefficients $\lambda_i$, for $i=1, \ldots, k$.

Now, given our family of ellipsoids $\mathcal{F}$, we can associate to each of them a point in $\rr^d \times \mathcal{Q}_d \cong \rr^{d(d+3)/2 - 1}$.  For every point $q \in \rr^d \times \mathcal{Q}_d$, let $D(q)$ be the ellispoid in $\rr^d$ represented by this pair.  Let $N = (r-1)\frac{d(d+3)}{2}$.  The points generated on $\rr^d \times \mathcal{Q}_d$ can be considered as a function from the vertices of $\Delta_N$ to $\rr^d \times \mathcal{Q}_d$.  We can extend this to a function $f: \Delta_N \to \rr^d \times \mathcal{Q}_d$ as described above in the construction of $(c,C)$.  Notice that this function is continuous.  Moreover, given vertices $v_1, \ldots, v_k$ of $\Delta_N$ and coefficients $\lambda_1, \ldots, \lambda_N$ of a convex combination, the considerations above show that
\[
D\left(f\left(\sum_{i=1}^k \lambda_i v_i\right)\right) \subset \conv \left(\bigcup_{i=1}^k D(f(v_i))\right).
\]
Therefore, if we apply the topological Tverberg theorem to the points generated in $\rr^d \times \mathcal{Q}_d$, the partition induced on our ellipsoids satisfies the conditions we wanted.
\end{proof}

\section{Final remarks and open problems}\label{section-remarks}

We have mentioned that our geometric Helly-type results have $(p,q)$-type analogues.  We present one for ellipsoids to illustrate the meaning of these generalizations clearly.

\begin{theorem}
	Let $p, q, d$ be positive integers such that $p \ge q > d(d+3)/2$.  Then, there exists a value $c = c(p,q,d)$ such that the following holds.  Let $\mathcal{F}$ be a finite family of convex sets in $\rr^d$, each of which contains an ellipsoid of volume one.  Suppose that for every choice of $p$ sets of $\mathcal{F}$, there exists $q$ of them whose intersection contains an ellipsoid of volume one.  Then, there exists $c$ ellipsoids of volume one such that every set in $\mathcal{F}$ contains at least one of them.
\end{theorem}

\subsection{Computational aspects of the Banach-Mazur distance}
Let $K, M$ be two convex sets and $\varepsilon$ be a positive real number.  Consider
\[
BM(K,M,\varepsilon) = \{(a,A) \in \rr^d \times \mathcal{P}_d : a +AK \subset M \subset a + (1+\varepsilon)AK \}
\]

This is closely related to the Banach-Mazur distance, but includes the additional point $a$.

\begin{problem}
Let $K, M \subset \rr^d$ be convex sets and $\varepsilon > 0$.  Determine if the set $BM(K,M,\varepsilon)$ is necessarily contractible.
\end{problem}

If we only require the condition $a+AK \subset M$, the set $BM(K,M,\varepsilon)$ would be convex.  If we only require the condition $M \subset a+(1+\varepsilon)AK$, the set $BM(K,M,\varepsilon)$ would be contractible.  We don't know if the intersection of these two sets is always contractible.  A positive answer would imply analogues of Theorem \ref{theorem-algorithm-approx} for the Banach-Mazur metric and simultaneous approximation of a family of sets by a single ellipsoid.

\subsection{Fractional Helly theorems}

Recently, Holmsen proved that fractional Helly theorems, which generalize Katchalski and Liu's classical result \cite{Katchalski:1979vt}, are a purely combinatorial consequence of the colorful version of Helly's theorem \cite{Holmsen:2019tw}.  Therefore, by applying his results directly to our colorful theorems we obtain results such as the following:

\begin{theorem}[Fractional Helly for ellipsoids of volume one]
	For every positive real number $\alpha$, there exists a $\beta >0$ that depends only on $\alpha$ and $d$ such that the following holds.  Let $n = d(d+3)/2$, and $\mathcal{F}$ be a finite family of convex sets in $\rr^d$.  If there are at least $\alpha {{|\mathcal{F}|}\choose{n}}$ subfamilies $\mathcal{G} \subset \mathcal{F}$ of size $n$ whose intersection contains an ellipsoid of volume one, then there exists a subfamily $\mathcal{F}' \subset \mathcal{F}$ such that $|\mathcal{F}'| \ge \beta |\mathcal{F}|$ and whose intersection contains an ellipsoid of volume one.
\end{theorem}

The result above implies a fractional Helly for the volume

\begin{theorem}[Fractional Helly for the volume]\label{theorem-fractinal-volume}
	For every positive real number $\alpha$, there exists a $\beta >0$ that depends only on $\alpha$ and $d$ such that the following holds.  Let $n = d(d+3)/2$, and $\mathcal{F}$ be a finite family of convex sets in $\rr^d$.  If there are at least $\alpha {{|\mathcal{F}|}\choose{n}}$ subfamilies $\mathcal{G} \subset \mathcal{F}$ of size $n$ whose intersection has volume greater than or equal to one, then there exists a subfamily $\mathcal{F}' \subset \mathcal{F}$ such that $|\mathcal{F}'| \ge \beta |\mathcal{F}|$ and whose intersection has volume greater than or equal to $d^{-d}$.
\end{theorem}

\begin{problem}
Does Theorem \ref{theorem-fractinal-volume} hold with $n=2d$?  The volume in the conclusion may need to be relaxed to $d^{-cd}$ for some constant $c>0$.
\end{problem}

It is not even clear to us that the fractional Helly theorems for the volume must have a loss factor.  In other words, for $n$ large enough, yet only dependent on $d$, it could be the case that in Theorem \ref{theorem-fractinal-volume} we can conclude that $\vol (\cap \mathcal{F}') \ge 1$.

\subsection{Further restrictions on the sets}

Classical variations of Helly's theorem on different sets mean that we replace the convexity condition on the family.  For example, Helly's theorem for boxes in $\rr^d$ means that we seek to guarantee that a family of parallel boxes intersects.  For this case, it is sufficient to check that every pair of boxes intersects: the Helly number is two.  The possible intersection structure of families of boxes is well understood \cite{Eckhoff:1988eo, Eckhoff:1991ue}.

In contrast, our results for boxes have Helly number $2d$, which is optimal.  The key difference is that we seek to contain boxes of certain volume, instead of restricting the sets themselves.  This can also be noted in the difference of our results for $H$-convex sets and the Helly theorems by Boltyanski and Martini \cite{Boltyanski:2003ir}.

However, it is natural to ask if the exact quantitative Helly theorems can be improved if we restrict the family $\mathcal{F}$ further.  For example, if $\mathcal{F}$ is a family of parallel boxes, if the intersection of every $2d$ or fewer sets of $\mathcal{F}$ of them contains an ellipsoid of volume one, then so does the intersection of all of them.  This can be seen because the volume ratio between a box and its John ellipsoid is constant, so the theorem boils down to our quantitative theorem for volume of boxes.

  The arguments above show that our results can be improved by imposing conditions on the sets.  It does not seem evident which properties of the sets of $\mathcal{F}$ and those of the witness sets can be combined to give improved Helly numbers.
  
  \subsection{Optimality of quantitative Helly and Tverberg theorems}
  
  For Tverberg theorems, it is not clear if the results we obtain are optimal.  For example, determining the optimal dependence of the dimension on our volumetric and diameter Tverberg theorems is of particular interest.
  
  In our volumetric Helly theorems, such as Corollary \ref{corollary-simpleellipsoidcoloring}, we do not know the optimal number of color classes.  It is possible that it is $2d$, which would give an honest colorful version of the optimal bounds of the B\'ar\'any-Katchalski-Pach theorem.
  
Our methods lift our families of sets in $\rr^d$ to a convex set in a higher dimension, $\rr^l$.  However, not every convex set in $\rr^l$ can be obtained by this construction.  For instance, consider Example \ref{example-translates}.  The set of possible translating vectors $a$ in the construction of $S_K(M)$ is the Minkowski difference of $M$ and $K$.  If the sets of convex sets we can obtain as $S_K(M)$ is restricted, it is likely that their intersection patterns will satisfy additional properties and improve some of our results.

\section{Acknowledgments}

The authors would like to thank Emo Welzl for his helpful comments and for pointing out the algorithmic applications of our results in the framework of LP-type problems.

%\section{Acknowlegments}

\newcommand{\etalchar}[1]{$^{#1}$}
\providecommand{\bysame}{\leavevmode\hbox to3em{\hrulefill}\thinspace}
\providecommand{\MR}{\relax\ifhmode\unskip\space\fi MR }
% \MRhref is called by the amsart/book/proc definition of \MR.
\providecommand{\MRhref}[2]{%
  \href{http://www.ams.org/mathscinet-getitem?mr=#1}{#2}
}
\providecommand{\href}[2]{#2}


\begin{thebibliography}{DLLHORP17}

\bibitem[ABDLL16]{Aliev:2016il}
Iskander Aliev, Robert Bassett, Jes{\'u}s~A. De~Loera, and Quentin Louveaux,
  \emph{{A quantitative Doignon-Bell-Scarf theorem}}, Combinatorica \textbf{37}
  (2016), no.~3, 313--332.

\bibitem[ADLS17]{Amenta:2017ed}
Nina Amenta, Jes{\'u}s~A. De~Loera, and Pablo Sober{\'o}n,
  \emph{{Helly{\textquoteright}s theorem: New variations and applications}},
  Contemporary Mathematics, vol. 685, American Mathematical Society,
  Providence, Rhode Island, March 2017.

\bibitem[AGMP{\etalchar{+}}17]{Averkov:2017ge}
Gennadiy Averkov, Bernardo Gonz{\'a}lez~Merino, Ingo Paschke, Matthias
  Schymura, and Stefan Weltge, \emph{{Tight bounds on discrete quantitative
  Helly numbers}}, Advances in Applied Mathematics \textbf{89} (2017), 76--101.

\bibitem[AK92]{Alon:1992gb}
Noga Alon and Daniel~J. Kleitman, \emph{{Piercing convex sets and the
  Hadwiger-Debrunner (p,q)-problem}}, Advances in Mathematics \textbf{96}
  (1992), no.~1, 103--112.

\bibitem[And79]{Ando:1979ju}
Tsuyoshi Ando, \emph{{Concavity of certain maps on positive definite matrices
  and applications to Hadamard products}}, Linear algebra and its applications
  \textbf{26} (1979), 203--241.

\bibitem[Bal97]{Ball:1997ud}
Keith~M. Ball, \emph{{An elementary introduction to modern convex geometry}},
  Flavors of geometry (1997).

\bibitem[B{\'a}r82]{Barany:1982va}
Imre B{\'a}r{\'a}ny, \emph{{A generalization of Carath{\'e}odory's theorem}},
  Discrete Mathematics \textbf{40} (1982), no.~2-3, 141--152.

\bibitem[BKP82]{Barany:1982ga}
Imre B{\'a}r{\'a}ny, Meir Katchalski, and J{\'a}nos Pach, \emph{{Quantitative
  Helly-type theorems}}, Proceedings of the American Mathematical Society
  \textbf{86} (1982), no.~1, 109--114.

\bibitem[BM03]{Boltyanski:2003ir}
Vladimir Boltyanski and Horst Martini, \emph{{Minkowski addition of H-convex
  sets and related Helly-type theorems}}, Journal of Combinatorial Theory,
  Series A \textbf{103} (2003), no.~2, 323--336.

\bibitem[Bor48]{Borsuk:1948tz}
Karol Borsuk, \emph{{On the imbedding of systems of compacta in simplicial
  complexes}}, Fundamenta Mathematicae (1948), no.~1, 217--234.

\bibitem[Bor75]{Borell:1975kx}
Christer Borell, \emph{{Convex set functions in d-space}}, Periodica
  Mathematica Hungarica \textbf{6} (1975), no.~2, 111--136.

\bibitem[Bra16]{Brazitikos:2016ja}
Silouanos Brazitikos, \emph{{Quantitative Helly-Type Theorem for the Diameter
  of Convex Sets}}, Discrete {\&} Computational Geometry \textbf{57} (2016),
  no.~2, 494--505.

\bibitem[Bra17]{Brazitikos:2017ts}
\bysame, \emph{{Brascamp{\textendash}Lieb inequality and quantitative versions
  of Helly's theorem}}, Mathematika \textbf{63} (2017), no.~1, 272--291.

\bibitem[BS18]{Barany:2018fy}
Imre B{\'a}r{\'a}ny and Pablo Sober{\'o}n, \emph{{Tverberg{\textquoteright}s
  theorem is 50 years old: A survey}}, Bulletin of the American Mathematical
  Society \textbf{55} (2018), no.~4, 459--492.

\bibitem[BSS81]{Barany:1981vh}
Imre B{\'a}r{\'a}ny, Senya~B. Shlosman, and Andr{\'a}s Sz{\"u}cs, \emph{{On a
  topological generalization of a theorem of Tverberg}}, Journal of the London
  Mathematical Society \textbf{2} (1981), no.~1, 158--164.

\bibitem[BZ17]{Blagojevic:2017bl}
Pavle V.~M. Blagojevi{\'c} and G{\"u}nter~M. Ziegler, \emph{{Beyond the
  Borsuk{\textendash}Ulam Theorem: The Topological Tverberg Story}}, A Journey
  Through Discrete Mathematics, Springer, Cham, Cham, 2017, pp.~273--341.

\bibitem[Car10]{Carlen:2010Ib}
Eric Carlen, \emph{{Trace inequalities and quantum entropy: an introductory
  course}}, Entropy and the quantum, Contemporary Mathematics 529, 2010,
  pp.~73--140.

\bibitem[CdVGG14]{ColindeVerdiere:2014gwa}
{\'E}ric Colin~de Verdi{\`e}re, Gr{\'e}gory Ginot, and Xavier Goaoc,
  \emph{{Helly numbers of acyclic families}}, Advances in Mathematics
  \textbf{253} (2014), 163--193.

\bibitem[Dam17]{Damasdi:2017ta}
G{\'a}bor Dam{\'a}sdi, \emph{{Some problems in combinatorial geometry}},
  Master's thesis, E{\"o}tv{\"o}s Lor{\'a}nd University, 2017.

\bibitem[DFN19]{Damasdi:2019vm}
G{\'a}bor Dam{\'a}sdi, Vikt{\'o}ria F{\"o}ldv{\'a}ri, and M{\'a}rton
  Nasz{\'o}di, \emph{{Colorful Helly-type Theorems for Ellipsoids}}, arXiv.org
  (2019).

\bibitem[DLGMM19]{DeLoera:2019jb}
Jes{\'u}s~A. De~Loera, Xavier Goaoc, Fr{\'e}d{\'e}ric Meunier, and Nabil~H.
  Mustafa, \emph{{The discrete yet ubiquitous theorems of Carath{\'e}odory,
  Helly, Sperner, Tucker, and Tverberg}}, Bulletin of the American Mathematical
  Society \textbf{56} (2019), no.~3, 1--97.

\bibitem[DLLHORP17]{DeLoera:2017th}
Jes{\'u}s~A. De~Loera, Reuben~N. La~Haye, Deborah Oliveros, and Edgardo
  Rold{\'a}n-Pensado, \emph{{Helly numbers of algebraic subsets of Rd and an
  extension of Doignon{\textquoteright}s Theorem}}, Advances in Geometry
  \textbf{17} (2017), no.~4, 473--482.

\bibitem[DLLHRS17a]{DeLoera:2017gt}
Jes{\'u}s~A. De~Loera, Reuben~N. La~Haye, David Rolnick, and Pablo Sober{\'o}n,
  \emph{{Quantitative Combinatorial Geometry for Continuous Parameters}},
  Discrete {\&} Computational Geometry \textbf{57} (2017), no.~2, 318--334.

\bibitem[DLLHRS17b]{DeLoera:2017bl}
\bysame, \emph{{Quantitative Tverberg Theorems Over Lattices and Other Discrete
  Sets}}, Discrete {\&} Computational Geometry \textbf{58} (2017), no.~2,
  435--448.

\bibitem[Eck88]{Eckhoff:1988eo}
J{\"u}rgen Eckhoff, \emph{{Intersection properties of boxes. Part I: An
  upper-bound theorem}}, Israel journal of mathematics \textbf{62} (1988),
  no.~3, 283--301.

\bibitem[Eck91]{Eckhoff:1991ue}
\bysame, \emph{{Intersection properties of boxes part II: Extremal families}},
  Israel journal of mathematics \textbf{73} (1991), no.~2, 129--149.

\bibitem[Fri15]{Frick:2015wp}
Florian Frick, \emph{{Counterexamples to the topological Tverberg conjecture}},
  arxiv.org (2015).

\bibitem[GMR{\v{S}}08]{Gartner:2008bp}
Bernd G{\"a}rtner, Ji{\v{r}}{\'\i} Matou{\v s}ek, Leo R{\"u}st, and Petr
  {\v{S}}kovro{\v{n}}, \emph{{Violator spaces: Structure and algorithms}},
  Discrete Applied Mathematics \textbf{156} (2008), no.~11, 2124--2141.

\bibitem[Gru88]{Gruber:1988gn}
Peter~M. Gruber, \emph{{Minimal ellipsoids and their duals}}, Rendiconti del
  Circolo Matematico di Palermo \textbf{37} (1988), no.~1, 35--64.

\bibitem[Gru08]{Gruber:2008gr}
\bysame, \emph{{Application of an idea of Vorono{\i} to John type problems}},
  Advances in Mathematics \textbf{218} (2008), no.~2, 309--351.

\bibitem[Gru11]{Gruber:2011dx}
\bysame, \emph{{John and Loewner Ellipsoids}}, Discrete {\&} Computational
  Geometry \textbf{46} (2011), no.~4, 776--788.

\bibitem[Hel23]{Helly:1923wr}
Eduard Helly, \emph{{{\"U}ber Mengen konvexer K{\"o}rper mit gemeinschaftlichen
  Punkte.}}, Jahresbericht der Deutschen Mathematiker-Vereinigung \textbf{32}
  (1923), 175--176.

\bibitem[Hia13]{Hiai:2013kz}
Fumio Hiai, \emph{{Concavity of certain matrix trace and norm functions}},
  Linear algebra and its applications \textbf{439} (2013), no.~5, 1568--1589.

\bibitem[HK17]{Holmsen:2017iw}
Andreas Holmsen and Roman~N. Karasev, \emph{{Colorful theorems for strong
  convexity}}, Proceedings of the American Mathematical Society \textbf{145}
  (2017), no.~6, 2713--2726.

\bibitem[Hol19]{Holmsen:2019tw}
Andreas~F. Holmsen, \emph{{Large cliques in hypergraphs with forbidden
  substructures}}, arxiv.org (2019).

\bibitem[HW17]{Holmsen:2017uf}
Andreas~F. Holmsen and Rephael Wenger, \emph{{HELLY-TYPE THEOREMS AND GEOMETRIC
  TRANSVERSALS}}, Handbook of Discrete and Computational Geometry,
  books.google.com, 2017, pp.~91--123.

\bibitem[KL79]{Katchalski:1979vt}
Meir Katchalski and Andy Liu, \emph{{A problem of geometry in $R^n$}},
  Proceedings of the American Mathematical Society, 1979.

\bibitem[KM05]{Kalai:2005tb}
Gil Kalai and Roy Meshulam, \emph{{A topological colorful Helly theorem}},
  Advances in Mathematics \textbf{191} (2005), no.~2, 305--311.

\bibitem[Ler50]{Leray:1950wn}
Jean Leray, \emph{{\emph{L{\textquoteright}anneau spectral et
  l{\textquoteright}anneau filtr{\'e} d{\textquoteright}homologie
  d{\textquoteright}un espace localement compact et d{\textquoteright}une
  application continue}}}, Journal de Math{\'e}matiques Pures et Appliqu{\'e}es
  \textbf{29} (1950), no.~9, 1--139.

\bibitem[Lie73]{Lieb:1973eu}
Elliott~H Lieb, \emph{{Convex trace functions and the Wigner-Yanase-Dyson
  conjecture}}, Advances in Mathematics \textbf{11} (1973), no.~3, 267--288.

\bibitem[MW15]{Mabillard:2015vx}
Isaac Mabillard and Uli Wagner, \emph{{Eliminating Higher-Multiplicity
  Intersections, I. A Whitney Trick for Tverberg-Type Problems}}, arXiv.org
  (2015).

\bibitem[Nas16]{Naszodi:2016he}
M{\'a}rton Nasz{\'o}di, \emph{{Proof of a Conjecture of B{\'a}r{\'a}ny,
  Katchalski and Pach}}, Discrete {\&} Computational Geometry \textbf{55}
  (2016), no.~1, 243--248.

\bibitem[Oxl06]{Oxley:2006uz}
James~G Oxley, \emph{{Matroid theory}}, Oxford University Press, 2006.

\bibitem[{\"O}za87]{Oza87}
Murad {\"O}zaydin, \emph{{Equivariant maps for the symmetric group}}, 1987.

\bibitem[RS17]{Rolnick:2017cm}
David Rolnick and Pablo Sober{\'o}n, \emph{{Quantitative (p,q)theorems in
  combinatorial geometry}}, Discrete Mathematics \textbf{340} (2017), no.~10,
  2516--2527.

\bibitem[Sob16]{Soberon:2016co}
Pablo Sober{\'o}n, \emph{{Helly-type theorems for the diameter}}, Bulletin of
  the London Mathematical Society \textbf{48} (2016), no.~4, 577--588.

\bibitem[SW92]{Sharir:1992ih}
Micha Sharir and Emo Welzl, \emph{{A combinatorial bound for linear programming
  and related problems}}, Annual Symposium on Theoretical Aspects of Computer
  Science (Berlin) (A~Finkel and M~Jantzen, eds.), Springer Berlin Heidelberg,
  1992, pp.~567--579.

\bibitem[SW14]{Saumard:2014fz}
Adrien Saumard and Jon~A Wellner, \emph{{Log-Concavity and Strong
  Log-Concavity: a review}}, Statistics surveys \textbf{8} (2014), no.~0,
  45--114.

\bibitem[Tan13]{Tancer:2013iza}
Martin Tancer, \emph{{Intersection Patterns of Convex Sets via Simplicial
  Complexes: A Survey}}, Thirty Essays on Geometric Graph Theory, Springer New
  York, New York, NY, 2013, pp.~521--540.

\bibitem[Tve66]{Tverberg:1966tb}
Helge Tverberg, \emph{{A generalization of Radon{\textquoteright}s theorem}},
  J. London Math. Soc \textbf{41} (1966), no.~1, 123--128.

\bibitem[Vol96]{Volovikov:1996up}
Alexey~Yu. Volovikov, \emph{{On a topological generalization of the Tverberg
  theorem}}, Mathematical Notes \textbf{59} (1996), no.~3, 324--326.

\end{thebibliography}
\end{document}